\documentclass[10pt,a4paper,oneside]{elsarticle}
 \usepackage[utf8]{inputenc}
 \usepackage{amsmath,amssymb}
 \usepackage{nameref,hyperref}
 \usepackage{color}
 \usepackage{graphicx}
 \date{}
 \usepackage{lastpage,fancyhdr,graphicx}
 \usepackage{epstopdf}
 \usepackage{epsfig}
 \usepackage{float}
 \usepackage{caption}
 \usepackage{subcaption}
 
\usepackage[a4paper, total={6in, 9in}]{geometry} 

 \newtheorem{theorem}{Theorem}[section]

 \newtheorem{definition}[theorem]{Definition}

\usepackage{xpatch}
\newenvironment{proof}{\paragraph{Proof:}}{\hfill$\square$}

\title{Harmful information spreading and its impact on vaccination campaigns modeled through fractal-fractional operators}
\author[aff1,aff2,aff3,aff4]{Ali Akg\"{u}l}
\author[aff5,aff6,aff7]{Auwalu Hamisu Usman}
\author[aff8]{J. Alberto Conejero\thanks{Corresponding author: aconejero@upv.es}}

\address[aff1]{Department of Electronics and Communication Engineering, Saveetha School of Engineering, SIMATS, Chennai, Tamilnadu, India.}
\address[aff2]{Siirt University, Art and Science Faculty, Department of Mathematics, 56100 Siirt, Turkey}
\address[aff3]{Department of Computer Engineering, Biruni University, 34010 Topkapı, Istanbul, Turkey} 
\address[aff4]{Near East University, Mathematics Research Center, Department of Mathematics, Near East Boulevard, PC: 99138, Nicosia /Mersin 10 – Turkey} 
\address[aff5]{Dept. Mathematics, Faculty of  Science, King Mongkut's University of Technology Thonburi, Bangkok, Thailand.}
\address[aff6]{KMUTTFixed Point Research Laboratory, Faculty of Science, King Mongkut’s University of Technology Thonburi (KMUTT), Bangkok 10140, Thailand.}
\address[aff7]{Dept. Math. Sci., Faculty of Physical Sciences, Bayero University, Kano. Nigeria.} 
\address[aff8]{Instituto de Matemática Pura y Aplicada. Universitat Politècnica de València, Spain.}

\begin{document}
 \maketitle
\begin{abstract}
Despite the huge efforts to develop and administer vaccines worldwide to cope with the COVID-19 pandemic, misinformation spreading through fake news in media and social networks about vaccination safety,
make that people refuse to be vaccinated, which harms not only these people but also the whole population.

In this work, we model the effects of harmful information spreading 
in immunization acquisition through vaccination. Our model is posed for several fractional derivative operators. We have conducted a comprehensive foundation analysis of this model for the different fractional derivatives. Additionally, we have incorporated a strength parameter that shows the combined impact of nonlinear and linear components within an epidemiological model. We have used the second derivative of the Lyapunov function to ascertain the detection of wave patterns within the vaccination dynamics. 
\end{abstract}


\section{Introduction} 	

In a interconnected world, one of the most significant challenges the rapid spreading of harmful information. Such information can modify perceptions, promote divisions or fear, and diminish the trust in institutions. Besides, they are used as a tool to unbalanced the results of political elections \cite{swire2017processing,hinds2020wouldn}.

Misinformation has a significant prevalence on health issues, such as vaccines, drugs or smoking,  noncommunicable diseases, pandemics, eating disorders, and medical treatments \cite{suarez2021prevalence}. 
When these narratives circulate unchecked, they are impacting on public health and safety policies. To fight against it, it is crucial to set a collective vigilance and to foster the critical thinking of all population, since we are immerse in a context where the truth and the falsehood often mixed and blurred.\medskip

It is imperative that we promote media literacy and support responsible communication, ensuring that accurate information prevails in the face of harmful narratives \cite{jones2021does,dame2022combating}.
One of the most popular cases that we have recently suffered was the spreading of harmful  information about the collateral effects of COVID19 vaccination \cite{alam2021fighting,broniatowski2021first,hansson2021covid}.
Such misinformation led to a reduction in the number of vaccinated people, despite it was confirmed the impact of vaccination campaigns to fight against the SARS-CoV2 virus spreading \cite{hernandez2022waning,begga2023predicting}.\medskip

Many mathematical models have been proposed to predict the spreading of the COVID19 disease around the world \cite{lozano2021open,ijcai2022p740,janko2023optimizing}. These forecasts have far-reaching implications for how quickly and forcefully governments respond to an epidemic. However, rather than producing precise quantitative estimates about the significant degree or timeframe of disease incidence, epidemiological models are best used to assess the relative impacts of multiple interventions in reducing the burden of diseases \cite{soltesz2020effect,hale2021global,perra2021non}.\medskip

Among the abundance of epidemiological models, many of them have been based on the use of fractional derivatives. Atangana and Araz proposed a mathematical model for COVID19 spread in Turkey and South Africa \cite{Atangana2020}. Boccaletti et al. investigated the modeling and forecasting of epidemic spreading, specifically focusing on COVID19 and beyond \cite{Boccaletti2020}. Atangana introduced a novel COVID19 model utilizing fractional differential operators with singular and non-singular kernels, presenting an analysis and numerical scheme based on Newton polynomial \cite{Atangana2021}. Omame et al. proposed a fractional-order model for COVID19 and tuberculosis co-infection using the Atangana–Baleanu derivative \cite{Omame2021}. Kolebaje et al. delved into the study of nonlinear growth and mathematical modeling of COVID-19 in several African countries utilizing the Atangana–Baleanu fractional derivative \cite{Kolebaje2022}. Some studies can be found in the following references \cite{Ahmed2021,Bertozzi,Amit,Rai}.\medskip

Due their widespread application, fractional differentiation and integration have become subjects of concern in a number of research papers. This topic has indeed attracted many researchers across all fields. The subject arose from a question posed by L'Hopital to Leibniz that initially raised the issue of exponential function differentiation \cite{leibniz1849letter}. Nowadays, many fractional operator have been proposed together with the applications 
 \cite{caputo1967linear,caputo2017notion,atangana2016new,chen2010anomalous,doungmo_goufo2016application}. The fractal derivative's application to fractal media has gotten a lot of attention. It can be used to model more complex physical problems using more complex differentiation mathematical operators \cite{tatom1995relationship,rocco1999fractional,
carpinteri2014fractals,butera2014physically,atangana2016new,kiryakova2017fractional}. 
\medskip

In this work we study several fractional models to simulate the effect of harmful information against vaccination campaigns. We first introduce, Caputo, Caputo Fabrizio, and Atangana-Baleanu.\medskip

\begin{definition}[Caputo Fractional Derivative]
Let $\alpha \in (0,1)$ and $f \in C^1([0,T])$. The Caputo fractional derivative of order $\alpha$ is defined as :
\[
{}^{C}D_t^{\alpha} f(t) = \frac{1}{\Gamma(1 - \alpha)} \int_0^t \frac{f'(s)}{(t - s)^{\alpha}} \, ds,
\]
where $\Gamma(\cdot)$ denotes the Gamma function.
\end{definition}

\begin{definition}[Caputo–Fabrizio Fractional Derivative]
Let $\alpha \in (0,1)$ and $f \in C^1([0,T])$. The Caputo–Fabrizio fractional derivative of order $\alpha$ is defined as \cite{caputo2015new}: 
\[
{}^{CF}D_t^{\alpha} f(t) = \frac{1 - \alpha}{1} f(t) + \frac{\alpha}{1 - \alpha} \int_0^t f'(s) \exp\left(-\frac{\alpha}{1 - \alpha}(t - s)\right) ds.
\]
\end{definition}

\begin{definition}[Atangana–Baleanu Fractional Derivative in Caputo Sense]
Let $\alpha \in (0,1)$ and $f \in C^1([0,T])$. The Atangana–Baleanu fractional derivative in the Caputo sense is defined as \cite{atangana2016new}:
\[
{}^{ABC}D_t^{\alpha} f(t) = \frac{B(\alpha)}{1 - \alpha} \int_0^t f'(s) E_{\alpha} \left(-\frac{\alpha}{1 - \alpha}(t - s)^{\alpha}\right) ds,
\]
where $E_{\alpha}(\cdot)$ is the Mittag-Leffler function and $B(\alpha)$ is a normalization function such that $B(0) = B(1) = 1$.
\end{definition}

We also consider the fractional differentiation operators with the next three kernels. These approach combines memory-dependence ($\alpha$) and with fractal geometry ($\eta$).

\begin{definition}[Fractal Fractional Derivatives]
Let $\alpha,\eta \in (0,1)$ and $f \in C^1([0,T])$. 
We can introduce three different fractional derivatives with power-law, exponential, and Mittag-Leffler kernels \cite{Ata17a}.

\begin{enumerate}
\item Powe-law type kernel.
	
	\begin{equation}
	^{FFP}_{c}\mathit{D}^{\alpha,\eta}_{t}f(t)=\frac{1}{1-\alpha}\frac{d}{du^{\eta}}\int_{c}^{t}f(s)(t-s)^{-\alpha}ds,~ 0<\alpha,\eta \le 1,
	\end{equation}
	where,
	\begin{equation}
	~~~~~\frac{df(s)}{ds^{\eta}}=\lim\limits_{t\to s}\frac{f(t)-f(s)}{t^{\eta}-s^{\eta}}~~~~~~~~~~~~~~~~~~~~~~~~~~~~~~~~~
	\end{equation}

\item Exponential-decay type kernel:
	\begin{equation}
	^{FFE}_{c}\mathit{D}^{\alpha,\eta}_{t}f(t)=\frac{\mathit{M_{1}}(\alpha)}{1-\alpha}\frac{d}{dt^{\eta}}\int_{c}^{t}f(s)\exp\Big(\frac{-\alpha}{1-\alpha}(t-s)\Big)ds,~0<\alpha,\eta \le 1.
	\end{equation}

\item Mittag-Leffler type kernel:
	
	\begin{equation}
	^{FFM}_{c}\mathit{D}^{\alpha,\eta}_{t}f(t)=\frac{\mathit{AB}(\alpha)}{1-\alpha}\frac{d}{dt^{\eta}}\int_{c}^{t}f(s)\mathit{E}_{\alpha}\Big(\frac{-\alpha}{1-\alpha}(t-s)^{\alpha}\Big)ds,~0<\alpha,\eta \le 1,
	\end{equation}
	where, $\mathit{AB}(\alpha)=1-\alpha+\frac{\alpha}{\Gamma(\alpha)}$.
\end{enumerate}

Let us also introduce the corresponding three integral operators:
\begin{enumerate}
\item Power-law type kernel:	
	\begin{equation}
	^{FFP}_{0}\mathit{I}^{\alpha,\eta}_{t}f(t)=\frac{\eta}{\Gamma(\alpha)}\int_{0}^{t}(t-s)^{\alpha-1}s^{\tau-1}\phi(s)ds.~~~~~~~~~~~~~~~~~~~~~~~~~~~~~~~~~
	\end{equation}
\item Exponential-decay type kernel
	\begin{equation}
	^{FFE}_{0}\mathit{I}^{\alpha,\eta}_{t}f(t)=\frac{{\alpha}\eta}{\mathit{M}_{1}(\alpha)}\int_{0}^{t}s^{\alpha-1}f(s)ds+\frac{\tau(1-\alpha)t^{\tau-1}}{M_{1}(\alpha)}\phi(t).~~~~~~~~~~~~~~~~~~
	\end{equation}
\item Mittag-Leffler type kernel
\begin{equation}
	^{FFM}_{0}\mathit{I}^{\alpha,\eta}_{t}f(t)=\frac{{\alpha}\eta}{\mathit{AB}(\alpha)}\int_{0}^{t}s^{\alpha-1}f(s)(t-s)^{\alpha-1}ds+\frac{\tau(1-\alpha)t^{\tau-1}}{\mathit{AB}(\alpha)}f(t).~~~~~~~~~~~~~~~
	\end{equation}
\end{enumerate}
\end{definition}

Beyond introducing a new model for the spreading of harmful information as a epidemiological model, we offer a new perspective to tackle the problem of harmful information spreading in our societies. We also provide a model through  fractal-fractional derivatives has never been analyzed so far. Additionally, we have discretized the fractal fractional model and conduct numerical simulations. We expect that this model open a new perspective to get deeper in the dynamics of SIR-type models.\medskip

The paper is organized as follows. In Section \ref{sec:formulation} we set the mathematical epidemiological model that describe the equations and variables in the model.
We first prove the existence and uniqueness of solutions in Section \ref{sec:ex_un}, 
In Section \ref{sec:positiveness_boundedness}, we show the positive boundedness for the classical model as long as for the different fractional models. We compute the equilibrium points in Section \ref{sec:equilibrium} and the reproduction and strength numbers in Section \ref{sec:rs_number}.
Finally, we illustrate the results of the solutions for fractal-fractional derivatives in Section \ref{sec:numerical} and outline some conclusions in Section \ref{sec:conclusions}.


\section{Mathematical formulation of harmful information impact on a vaccination campaign}
\label{sec:formulation}

We construct a mathematical model to portray the impact of harmful criticism on vaccines and how this decreases the effect of vaccination efforts to control the spreading.\medskip

Let us define the key variables:
$\Pi$ represents the influx or rate of individuals becoming eligible for a vaccine dose, while $N$ represents the total number of eligible individuals. 
$S_p(t)$ denotes the group of individuals willing to accept the vaccine administration but susceptible to being influenced by harmful criticisms and challenges. $I(t)$ represents the group affected by detrimental news media and criticisms. Within this group, $I_p(t)$ signifies individuals affected but still holding positive opinions regarding vaccination, while $I_n(t)$ refers to those affected who have developed negative views about it. Meanwhile, $I_c(t)$ indicates individuals affected and experiencing confusion, lacking clear positive or negative opinions about vaccination.
$R(t)$ encompasses individuals who have overcome divisions caused by criticisms and challenges.
Lastly, $D(t)$ accounts for individuals who succumb to the divisive effects, leading to either death or denial of vaccination.\medskip

Through this model, we aim to capture the dynamics and consequences of harmful criticisms on people's perceptions and decisions regarding a vaccination strategy, shedding light on how different groups respond and evolve. The model is given by the next equations
\begin{align}
\label{eqn1}
\frac{d S_p}{d t} & =\Pi-\frac{{\beta}S_pI}{N}-{\sigma}S_p-{\nu}S_p\\
\label{eqn2}
\frac{d I}{d t} &=\frac{{\beta}S_pI}{N}-\gamma_1I-\gamma_2I-\gamma_3I-\gamma_4I-\tau_1I-{\nu}I\\
\label{eqn3}
\frac{d I_p}{d t}& =\gamma_1I-\tau_3I_p-{\nu}I_p\\
\label{eqn4}
\frac{d I_n}{d t} & =\gamma_2I-\tau_2I_n-\phi_1I_n-{\nu}I_n\\
\label{eqn5}
\frac{d I_c}{d t} & =\gamma_3I-\tau_4{I_c}-\phi_2I_c-{\nu}I_c\\
\label{eqn6}
\frac{d R}{d t}& =\tau_1I+\tau_2I_n+\tau_3I_p+\tau_4I_c-\tau{R}-{\nu}R\\
\label{eqn7}
\frac{d D}{d t} &={\sigma}S_p+\phi_1I_n+\phi_2I_c+\gamma_4{I}+\tau{R}-{\nu}D.
\end{align}
with $(S_p^0,I^0,I_p^0,I_n^0,I_c^0,R^0,D^0)$ as initial conditions at $t=0$. The following diagram explains the flows modeled by the equations.

\begin{figure}[H]
\begin{center}
\includegraphics[scale=.4]{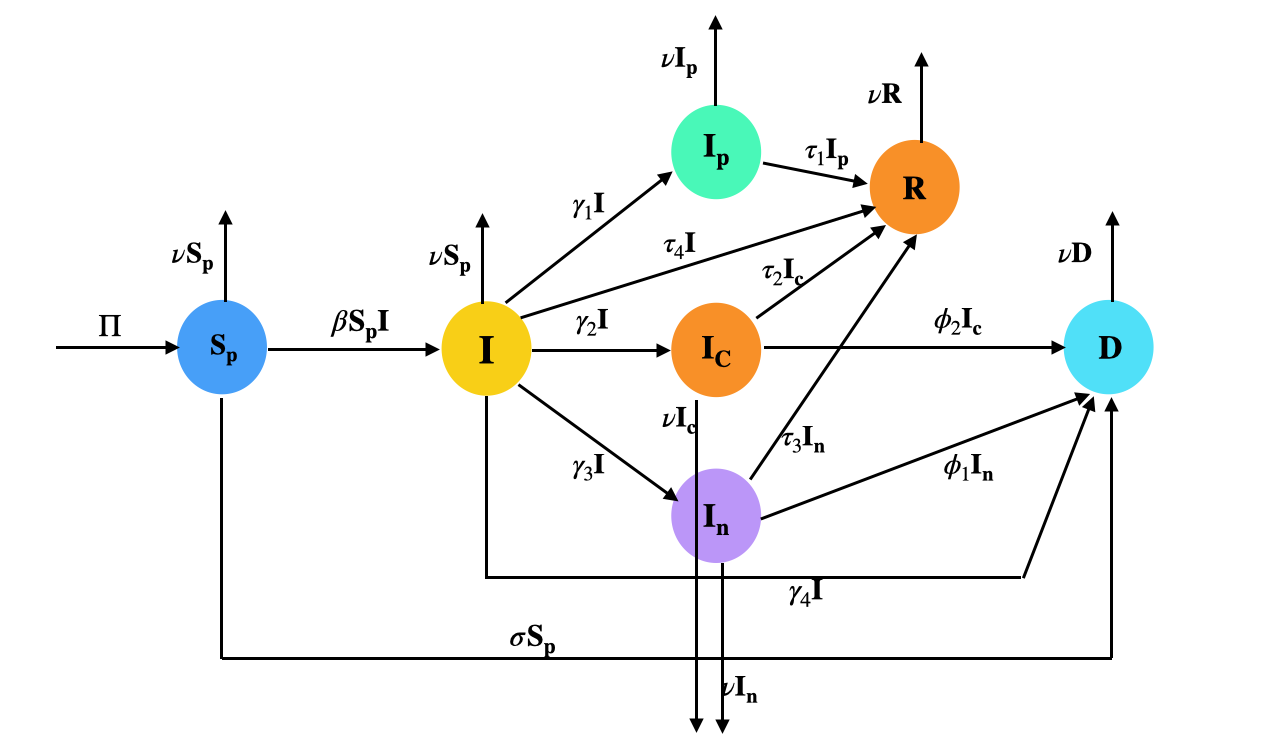}
\end{center}
\caption{Harmful information spreading model.}
\end{figure}

\section{Existence and Uniqueness}\label{sec:ex_un}
In this section, our aim is to conduct an extensive analysis establishing the existence and uniqueness of the solutions to the system of equations \eqref{eqn1}-\eqref{eqn7}. To achieve this objective, and since the right side of these equations are continuous, we only need to state the following theorem which proofs the property of being locally Lipschitz. The statement for the fractional derivatives is similar.

\begin{theorem}
Let us consider the model 
\begin{equation}
\left(\frac{d S_p}{d t},\frac{d I}{d t},\frac{d I_p}{d t},\frac{d I_n}{d t},\frac{d I_c}{d t},\frac{d R}{d t},\frac{d D}{d t}\right)=(G_i(S_p,I,I_p,I_n,I_c,R,D))_{1\le i\le 7}
\end{equation}
where the functions $G_i$ represent the left side of equations \eqref{eqn1}-\eqref{eqn7}.
There are positive constants $\rho_i$ and $\bar{\rho_i}$ such that for all $(x_i,t)\in\mathbb{R}^7\times[0,T]$ and for all $1\le i\le 7$, we have
\begin{align}
	& \;\big|G_i(x_i^1,t)-G_i(x_i^2,t)\big|^2\leqslant\rho_i\big|x_i^1-x_i^2\big|^2.\label{eqa}\\
	& \;\big|G_i(x_i,t)\big|^2\leqslant\bar{\rho_i}\big(1+|x_i|^2\big)^2 \label{eqb}.
	\end{align}
\end{theorem}

\begin{proof}
We start computing the estimations indicated in \eqref{eqa}.
\begin{align}
\big|G_1(S_p^1,t)-G_1(S_p^2,t)\big|^2&\nonumber =\bigg|\Pi-\left(\frac{{\beta}I}{N}+{\sigma}+{\nu}\right)S_p^1-\left(\Pi-\left(\frac{{\beta}I}{N}+{\sigma}+{\nu}\right)S_p^2\right)\bigg|^2\\ &\nonumber \leqslant \big|\left({{\beta}I}+{\sigma}+{\nu}\right)(S_p^1-S_p^2)\big|^2 \\ & \leqslant \rho_1\big|S_p^1-S_p^2\big|^2,\,\,\,\,\text{with}\,\,\,\, \rho_1= 2\left({{\beta^2}\|I\|_\infty^2}+{\sigma^4}+{\nu^2}\right).
\end{align}

In a similar way, we get estimations for the other of equations.
\begin{align}
\big|G_2(I^1,t)-G_2(I^2,t)\big|^2 & \le \rho_2|I^1-I^2|^2,
\end{align}
with $\rho_2=2\left({{\beta^2}\|S_p\|_\infty^2}+\sum_{i=1}^4\gamma_i^2+\tau_1^2+{\nu}^2\right)$.

\begin{align}
\big|G_3(I^1_p,t)-G_3(I^2_p,t)\big|^2 & \le \rho_3|I^1_p-I^2_p|^2,\,\,\,\, \text{with}\,\,\,\, \rho_3=2(\tau_3^2+{\nu^2})\\
\big|G_4(I^1_n,t)-G_4(I^2_n,t)\big|^2 & \le \rho_4|I^1_n-I^2_n|^2,\,\,\,\, \text{with}\,\,\,\, \rho_4=2(\tau_2^2+\phi_1^2+{\nu^2})\\
\big|G_5(I^1_c,t)-G_5(I^2_c,t)\big|^2 & \le \rho_5|I^1_c-I^2_c|^2,\,\,\,\, \text{with}\,\,\,\, \rho_5=2(\tau_4^2+\phi_2^2+{\nu^2})\\
\big|G_6(R^1,t)-G_6(R^2,t)\big|^2 & \le \rho_6|R^1-R^2|^2,\,\,\,\, \text{with}\,\,\,\, \rho_6=2(\tau^2+{\nu^2})\\
\big|G_7(D^1,t)-G_7(D^2,t)\big|^2 & \le \rho_7|D^1-D^2|^2,\,\,\,\, \text{with}\,\,\,\, \rho_7=2({\nu^2}+\epsilon_1)
\end{align}

To prove \eqref{eqb}, we start with the first equation

\begin{align}
\big|G_1(S_p,t)\big|^2 &\nonumber = \bigg|\Pi-\left(\frac{{\beta}I}{N}+{\sigma}+{\nu}\right)S_p\bigg|^2  \leqslant 2|\Pi|^2+2\left({{\beta^2}\|I\|_\infty^2}+{\sigma^2}^2+{\nu^2}\right)\big|S_p\big|^2 \\ &\nonumber \leqslant 2|\Pi|^2\left(1+\frac{2\left({{\beta^2}\|I\|_\infty^2}+{\sigma^2}^2+{\nu^2}\right)}{2|\Pi|^2}\big|S_p\big|^2\right) \leqslant \bar{\rho_1}\left(1+\big|S_p\big|^2\right)
\end{align}
with the condition $\bar{\rho_1}=1$ and $K_1=\frac{2\left({{\beta^2}\|I\|_\infty^2}+{\sigma^2}^2+{\nu^2}\right)}{2|\Pi|^2}<{1}$.

So that, we can proceed with the other equations as follows:

\begin{align}
\big|G_2(I,t)\big|^2 &\nonumber  =\bigg|\left(\frac{{\beta}S_p}{N}-(\gamma_1+\gamma_2+\gamma_3+\gamma_4+\tau_1+{\nu})I\right)\bigg|^2\leqslant \bar{\rho_2}\big(1+|I|^2\big),
\end{align}
with $\bar{\rho_2}=2\left({{\beta^2}\|S_p\|^2_\infty}+\gamma_1^2+\gamma_2^2+\gamma_3^2+\gamma_4^2+\tau_1^2+{\nu^2}\right)$.
\begin{align}
\big|G_3(I_p,t)\big|^2\big|^2&\nonumber =\left|\gamma_1I-(\tau_3+{\nu})I_p\right|^2 \leqslant 2(\gamma_1^2\|I\|_\infty^2)\left(1+\frac{2(\tau_3^2+{\nu^2})}{2(\gamma_1^2\|I\|_\infty^2)}\left|I_p\right|^2\right)\leqslant \bar{\rho_3}\left(1+|I_p^2|\right)
\end{align}
with 
$\bar{\rho_3}=2(\gamma_1^2\|I\|_\infty^2)$ and the condition $K_2=\frac{2(\tau_3^2+{\nu^2})}{2(\gamma_1^2\|I\|_\infty^2)}<1$

\begin{align}
|G_4(I_n,t)|^2&\nonumber =\big|\gamma_2I-(\tau_2+\phi_1+{\nu})I_n\big|^2 \leqslant 2\left(\gamma_2^2\|I\|_\infty^2\|I_c\|_\infty^2\right)\left(1+\frac{2(\tau_2+\phi_1+{\nu})}{2\left(\gamma_2^2\|I\|_\infty^2\|I_c\|_\infty^2\right)}\big|I_n\big|^2\right),\\ & \leqslant \bar{\rho_4}\left(1+\big|I_n\big|^2\right),
\end{align}
with $\bar{\rho_4}=2\left(\gamma_2^2\|I\|_\infty^2\|I_c\|_\infty^2\right)$ and the condition $K_3=\frac{2(\tau_2+\phi_1+{\nu})}{2\left(\gamma_2^2\|I\|_\infty^2\|I_c\|_\infty^2\right)}<1$
\begin{align}
\big|G_5(I_c,t)\big|^2&\nonumber \leqslant 2\gamma_3^2\|I\|^2_\infty\left(1+\frac{2(\tau_4^2+\phi_2^2+{\nu})}{2\gamma_3^2\|I\|^2_\infty}\big|I_c\big|^2\right),\\ & \leqslant \bar{\rho_5}\left(1+\big|I_c\big|^2\right),
\end{align}
with $\bar{\rho_5}=2\gamma_3^2\|I\|^2_\infty\,\,\,$ and $K_4=\frac{2(\tau_4^2+\phi_2^2+{\nu})}{2\gamma_3^2\|I\|^2_\infty}<1$
\begin{align}
\big|G_6(R,t)\big|^2 &\le 2\left(\tau^2_1\|I\|_\infty^2+\tau_2^2\|I_n\|_\infty^2+\tau_3^2\|I_p\|_\infty^2+\tau_4^2\|I_c\|_\infty^2\right)+2\left(\tau^2+{\nu^2}\right)|R|^2,\\ &\nonumber \leqslant \bar{\rho_6}\left(1+\frac{2\left(\tau^2+{\nu^2}\right)}{2\left(\tau^2_1\|I\|_\infty^2+\tau_2^2\|I_n\|_\infty^2+\tau_3^2\|I_p\|_\infty^2+\tau_4^2\|I_c\|_\infty^2\right)}|R|^2\right) \leqslant \bar{\rho_6}\left(1+|R|^2\right),
\end{align}
with $\bar{\rho_6}=2\left(\tau^2_1\|I\|_\infty^2+\tau_2^2\|I_n\|_\infty^2+\tau_3^2\|I_p\|_\infty^2+\tau_4^2\|I_c\|_\infty^2\right)$\\ and $K_5=\frac{2\left(\tau^2+{\nu^2}\right)}{2\left(\tau^2_1\|I\|_\infty^2+\tau_2^2\|I_n\|_\infty^2+\tau_3^2\|I_p\|_\infty^2+\tau_4^2\|I_c\|_\infty^2\right)}<1$
\begin{align}
\big|G_7(D,t)\big|^2 & \leqslant 2\left({\sigma^2}\|S_p\|_\infty^2+\phi_1^2\|I_n\|_\infty^2+\phi^2_2\|I_c\|_\infty^2+\gamma_4^2\|I\|_\infty^2+\tau^2\|R\|_\infty^2+{\nu^2}\big|D\big|^2\right)\\ &\nonumber \leqslant \bar{\rho_7}\left(1+\frac{2{\nu^2}}{2\left({\sigma^2}\|S_p\|_\infty^2+\phi_1^2\|I_n\|_\infty^2+\phi^2_2\|I_c\|_\infty^2+\gamma_4^2\|I\|_\infty^2+\tau^2\|R\|_\infty^2\right)}\big|D\big|^2\right)\leqslant \bar{\rho_7}\left(1+\big|D\big|^2\right),
\end{align}
$\bar{\rho_7}=2\left({\sigma^2}\|S_p\|_\infty^2+\phi_1^2\|I_n\|_\infty^2+\phi^2_2\|I_c\|_\infty^2+\gamma_4^2\|I\|_\infty^2+\tau^2\|R\|_\infty^2\right)\,\,\,$\\ with the condition that $K_6=\frac{2{\nu^2}}{2\left({\sigma^2}\|S_p\|_\infty^2+\phi_1^2\|I_n\|_\infty^2+\phi^2_2\|I_c\|_\infty^2+\gamma_4^2\|I\|_\infty^2+\tau^2\|R\|_\infty^2\right)}<1$.\medskip

Then, our system's solution exists and is unique provided that $\max\{K_1,\ldots K_6\}\le 1$.\\
\end{proof}

\section{Positiveness and boundedness of the solution}\label{sec:positiveness_boundedness}
We start with the study of the positivity and lower boundedness of the solutions for the abstract Cauchy problem posed in Section \ref{sec:formulation}.
Let us define the supremum norm of a function $f(t)$, with domain $D_f$,  by $
 \|f\| _\infty = \sup_{t \in D_f} |f(t)|$. Starting from
\begin{equation*}
\frac{d S_p}{d t} =\Pi-\frac{{\beta}S_pI}{N}-{\sigma}S_p-{\nu}S_p,
\geqslant-\left(\frac{{\beta}I}{N}+{\sigma}+{\nu}\right)S_p \geqslant-\left(\frac{{\beta}\|I\|_\infty}{\|N\|_\infty}+{\sigma}+{\nu}\right)S_p,
\end{equation*}
we can obtain a lower bound $
S_p(t)\geqslant S_p^0e^{-\left(\frac{{\beta}\|I\|_\infty}{\|N\|_\infty}+{\sigma}+{\nu}\right)t}$ for all $t\ge 0$. Applying the same approach for all the remaining functions, we obtain the following lower bounds for all $t\ge 0$:

\begin{center}
\begin{tabular}{lll}
$I(t)\geqslant I^0e^{-\left(\gamma_1+\gamma_2+\gamma_3+\gamma_4+\tau_1+{\nu}\right)t}$,
& $I_p(t)\geqslant I_p^0e^{-\left(\tau_3+{\nu}\right)t}$ 
& $I_n(t)\geqslant I_n^0e^{-\left(\tau_2+\phi_1+{\nu}\right)t}$\\ 
$I_c(t)\geqslant I_c^0e^{-\left(\tau_4+\phi_2+{\nu}\right)t}$
& $R(t)\geqslant R^0e^{-\left(\tau{R}+{\nu}\right)t}$ 
& $D(t)\geqslant D^0e^{-\left({\nu}\right)t}$ \\ 
\end{tabular} 
\end{center}

Now, let us study the positivity of the solutions for the fractional versions (Caputo, Caputo-Fabrizio, and Atangana-Baleanu) of the abstract Cauchy problem posed in \eqref{eqn1}-\eqref{eqn7}.

\subsection{The Caputo fractional derivative case}
We start with the Caputo fractional derivative. Since
\begin{align}
	& ^{C}_{0}D^{\alpha}_{t}S_p(t)\geqslant -\left(\frac{{\beta}\|I\|_\infty}{\|N\|_\infty}+{\sigma}+{\nu}\right)S_p(t), \text{ for all } t\geqslant{0},
\end{align}

Taking Laplace transformations to both sides, we get
\begin{equation}
\mathcal{L}\left[\frac{1}{\Gamma(1-\alpha)}\int_{0}^{t}S'_p(\tau)(t-\tau)^{-\alpha}d\tau\right]\geqslant \mathcal{L}\left[-\left(\frac{{\beta}\|I\|_\infty}{\|N\|_\infty}+{\sigma}+{\nu}\right)S_p(t)\right]
\end{equation}
and, using convolution and Laplace transformation properties we have
\begin{equation}
\nonumber\mathcal{L}\left[S'_p(t)\right]*\mathcal{L}\left[\frac{t^{-\alpha}}{\Gamma(1-\alpha)}\right]\geqslant -\left(\frac{{\beta}\|I\|_\infty}{\|N\|_\infty}+{\sigma}+{\nu}\right)\mathcal{L}\left[S_p(t)\right]
\end{equation}
\begin{equation}
\nonumber\left[S\mathcal{L}(S_p(t))-S_p(0)\right]{S^{\alpha-1}}\geqslant -\left(\frac{{\beta}\|I\|_\infty}{\|N\|_\infty}+{\sigma}+{\nu}\right)\mathcal{L}\left[S_p(t)\right]
\end{equation}
which gives
\begin{equation}
\mathcal{L}(S_p(t))\geqslant \frac{S^{\alpha-1}}{S^{\alpha}+\left(\frac{{\beta}\|I\|_\infty}{\|N\|_\infty}+{\sigma}+{\nu}\right)}S_p(0).
\end{equation} 

Finally, taking inverse transforms to both sides, we get
\begin{equation}
S_p(t)\geqslant S_p(0)E_\alpha\left[-\left(\frac{{\beta}\|I\|_\infty}{\|N\|_\infty}+{\sigma}+{\nu}\right)t^{\alpha}\right]\text{ for all } t\geqslant{0},
\end{equation}
where $E_\alpha(t)=\sum_{k=0}^{\infty} \frac{t^k}{\Gamma(1+\alpha k)}$ is the Mittag-Leffler function. Following the same procedure, one can get similar estimations for the remaining functions as
\begin{center}
\begin{tabular}{ll}
$I(t)\geqslant I(0)E_\alpha\left[-\left(\gamma_1+\gamma_2+\gamma_3+\gamma_4+\tau_1+{\nu}\right)t^{\alpha}\right]$,
& $I_p(t)\geqslant I_p(0)E_\alpha\left[-\left(\tau_3+{\nu}\right)t^{\alpha}\right]$,\\
$I_n(t)\geqslant I_n(0)E_\alpha\left[-\left(\tau_2+\phi_1+{\nu}\right)t^{\alpha}\right]$,
& $I_c(t)\geqslant I_c(0)E_\alpha\left[-\left(\tau_4+\phi_2+{\nu}\right)t^{\alpha}\right]$,\\
$R(t)\geqslant R(0)E_\alpha\left[-\left(\tau{R}+{\nu}\right)t^{\alpha}\right]$,
& $D(t)\geqslant D(0)E_\alpha\left[-\left({\nu}\right)t^{\alpha}\right]$, 
\end{tabular}
\end{center}
for all $t\ge 0$. 

\subsection{The Caputo-Fabrizio fractional derivative case}

Analogously, we can get similar estimations for the fractional version of the problem \eqref{eqn1}-\eqref{eqn7}. Let us start with the Caputo–Fabrizio fractional derivative. Taking \eqref{eqn1} and replacing the classical derivative with this one, we get
\begin{equation*}
^{CF}_{0}D^{\alpha}_{t}S_p(t)\geqslant -\left(\frac{{\beta}\|I\|_\infty}{\|N\|_\infty}+{\sigma}+{\nu}\right)S_p(t), \,\,\,\forall t\geqslant{0},
\end{equation*}
and taking Laplace transforms to both sides, we get
\begin{equation*}
\mathcal{L}\left[\frac{M(\alpha)}{(1-\alpha)}\int_{0}^{t}S'_p(\tau)Exp\left[-\frac{\alpha}{1-\alpha}(t-\tau)\right]d\tau\right]\geqslant \mathcal{L}\left[-\left(\frac{{\beta}\|I\|_\infty}{\|N\|_\infty}+{\sigma}+{\nu}\right)S_p(t)\right].
\end{equation*}

Then, using convolution and Laplace transformation properties, we can write
\begin{align*}
&\nonumber\mathcal{L}\left[S'_p(t)\right]*\mathcal{L}\left[\frac{M(\alpha)}{(1-\alpha)}Exp\left(-\frac{\alpha{t}}{1-\alpha}\right)\right]\geqslant -\left(\frac{{\beta}\|I\|_\infty}{\|N\|_\infty}+{\sigma}+{\nu}\right)\mathcal{L}\left[S_p(t)\right]\\
& \nonumber\left[S\mathcal{L}(S_p(t))-S_p(0)\right]\left(\frac{M(\alpha)}{(1-\alpha)}\right)\left(\frac{1-\alpha}{(1-\alpha)S+\alpha}\right)\geqslant -\left(\frac{{\beta}\|I\|_\infty}{\|N\|_\infty}+{\sigma}+{\nu}\right)\mathcal{L}\left[S_p(t)\right],\\
&\nonumber \mathcal{L}(S_p(t))\geqslant \frac{M(\alpha)}{SM(\alpha)+S(1-\alpha)\left(\frac{{\beta}\|I\|_\infty}{\|N\|_\infty}+{\sigma}+{\nu}\right)+\alpha\left(\frac{{\beta}\|I\|_\infty}{\|N\|_\infty}+{\sigma}+{\nu}\right)}S_p(0),\\
&\nonumber S_p(t)\geqslant \mathcal{L}^{-1}\left[\frac{M(\alpha)}{SM(\alpha)+S(1-\alpha)\left(\frac{{\beta}\|I\|_\infty}{\|N\|_\infty}+{\sigma}+{\nu}\right)+\alpha\left(\frac{{\beta}\|I\|_\infty}{\|N\|_\infty}+{\sigma}+{\nu}\right)}S_p(0)\right],\\
& S_p(t)\geqslant S_p(0)Exp\left[-\frac{\alpha\left(\frac{{\beta}\|I\|_\infty}{\|N\|_\infty}+{\sigma}+{\nu}\right)t}{M(\alpha)-(1-\alpha)\left(\frac{{\beta}\|I\|_\infty}{\|N\|_\infty}+{\sigma}+{\nu}\right)}\right] \text{ for all } t\geqslant{0}.
\end{align*}

Following the same procedures one can write for the other equations we get
\begin{center}
\begin{tabular}{ll}
$I(t)\geqslant I(0)Exp\left[\frac{-\alpha\left(\gamma_1+\gamma_2+\gamma_3+\gamma_4+\tau_1+{\nu}\right)t}{M(\alpha)-(1-\alpha)\left(\gamma_1+\gamma_2+\gamma_3+\gamma_4+\tau_1+{\nu}\right)}\right]$,
& $I_p(t)\geqslant I_p(0)Exp\left[\frac{-\left(\tau_3+{\nu}\right)t}{M(\alpha)-(1-\alpha)\left(\tau_3+{\nu}\right)}\right]$,\\
$I_n(t)\geqslant I_n(0)Exp\left[\frac{-\left(\tau_2+\phi_1+{\nu}\right)t}{M(\alpha)-(1-\alpha)\left(\tau_2+\phi_1+{\nu}\right)}\right]$,
& $I_c(t)\geqslant I_c(0)Exp\left[\frac{-\alpha\left(\tau_4+\phi_2+{\nu}\right)t}{M(\alpha)-(1-\alpha)\left(\tau_4+\phi_2+{\nu}\right)}\right]$,\\
$R(t)\geqslant R(0)Exp\left[-\frac{\left(\tau{R}+{\nu}\right)t}{M(\alpha)-(1-\alpha)\left({\nu}\right)}\right]$,
& $D(t)\geqslant D(0)Exp\left[-\frac{\alpha\left({\nu}\right)t}{M(\alpha)-(1-\alpha)\left({\nu}\right)}\right], \text{ for all } t\geqslant{0}$.
\end{tabular}
\end{center}\medskip

\subsection{The Atangana-Baleanu fractional derivative case}
We can repeat the same approach with the other kernels (power-law, exponential, and Mittag-Leffler kernels) obtaining the following estimations. For the Atangana-Baleanu fractional derivative:

\begin{center}
\begin{tabular}{ll}
$S_p(t)\geqslant S_p(0)E_\alpha\left[\frac{-\alpha\left(\frac{{\beta}\|I\|_\infty}{\|N\|_\infty}+{\sigma}+{\nu}\right)t^\alpha}{AB(\alpha)-(1-\alpha)\left(\frac{{\beta}\|I\|_\infty}{\|N\|_\infty}+{\sigma}+{\nu}\right)}\right]$,
& $I(t)\geqslant I(0)E_\alpha\left[\frac{-\alpha\left(\gamma_1+\gamma_2+\gamma_3+\gamma_4+\tau_1+{\nu}\right)t^\alpha}{AB(\alpha)-(1-\alpha)\left(\gamma_1+\gamma_2+\gamma_3+\gamma_4+\tau_1+{\nu}\right)}\right]$,\\
$I_p(t)\geqslant I_p(0)E_\alpha\left[\frac{-\left(\tau_3+{\nu}\right)t^\alpha}{AB(\alpha)-(1-\alpha)\left(\tau_3+{\nu}\right)}\right]$,
& $I_n(t)\geqslant I_n(0)E_\alpha\left[\frac{-\left(\tau_2+\phi_1+{\nu}\right)t^\alpha}{AB(\alpha)-(1-\alpha)\left(\tau_2+\phi_1+{\nu}\right)}\right]$,\\
$I_c(t)\geqslant I_c(0)E_\alpha\left[\frac{-\alpha\left(\tau_4+\phi_2+{\nu}\right)t^\alpha}{AB(\alpha)-(1-\alpha)\left(\tau_4+\phi_2+{\nu}\right)}\right]$,
& $R(t)\geqslant R(0)E_\alpha\left[-\frac{\left(\tau{R}+{\nu}\right)t^\alpha}{AB(\alpha)-(1-\alpha)\left({\nu}\right)}\right]$,\\
$D(t)\geqslant D(0)E_\alpha\left[-\frac{\alpha\left({\nu}\right)t^\alpha}{AB(\alpha)-(1-\alpha)\left({\nu}\right)}\right]$ for all $ t\ge 0$.
\end{tabular}
\end{center}

Repeating the procedure for fractal-fractional operators and denoting by $c$ the time component, we can get the following estimations for the power-law, exponential, and Mittag-Leffler kernels. For the power-law kernel we get:

\subsection{The fractal fractional derivative case}
\begin{center}
\begin{tabular}{ll}
$S_p(t)\geqslant S_p(0)E_\alpha\left[-c^{1-\beta}\left(\frac{{\beta}\|I\|_\infty}{\|N\|_\infty}+{\sigma}+{\nu}\right)t^{\alpha}\right]$
& $I(t)\geqslant I(0)E_\alpha\left[-c^{1-\beta}\left(\gamma_1+\gamma_2+\gamma_3+\gamma_4+\tau_1+{\nu}\right)t^{\alpha}\right]$,\\
$I_p(t)\geqslant I_p(0)E_\alpha\left[-c^{1-\beta}\left(\tau_3+{\nu}\right)t^{\alpha}\right]$,
& $I_n(t)\geqslant I_n(0)E_\alpha\left[-c^{1-\beta}\left(\tau_2+\phi_1+{\nu}\right)t^{\alpha}\right]$\\
$I_c(t)\geqslant I_c(0)E_\alpha\left[-c^{1-\beta}\left(\tau_4+\phi_2+{\nu}\right)t^{\alpha}\right]$,
& $R(t)\geqslant R(0)E_\alpha\left[-c^{1-\beta}\left(\tau{R}+{\nu}\right)t^{\alpha}\right]$,\\
$D(t)\geqslant D(0)E_\alpha\left[-c^{1-\beta}\left({\nu}\right)t^{\alpha}\right]$ for all $t\geqslant{0}$.
\end{tabular}
\end{center}

For an exponential kernel we get
\begin{align*}
&S_p(t)\geqslant S_p(0)Exp\left[-\frac{c^{1-\beta}\alpha\left(\frac{{\beta}\|I\|_\infty}{\|N\|_\infty}+{\sigma}+{\nu}\right)t}{M(\alpha)-(1-\alpha)\left(\frac{{\beta}\|I\|_\infty}{\|N\|_\infty}+{\sigma}+{\nu}\right)}\right],\\
&I(t)\geqslant I(0)Exp\left[\frac{-c^{1-\beta}\alpha\left(\gamma_1+\gamma_2+\gamma_3+\gamma_4+\tau_1+{\nu}\right)t}{M(\alpha)-(1-\alpha)\left(\gamma_1+\gamma_2+\gamma_3+\gamma_4+\tau_1+{\nu}\right)}\right],\\
&I_p(t)\geqslant I_p(0)Exp\left[\frac{-c^{1-\beta}\left(\tau_3+{\nu}\right)t}{M(\alpha)-(1-\alpha)\left(\tau_3+{\nu}\right)}\right],\\
&I_n(t)\geqslant I_n(0)Exp\left[\frac{-c^{1-\beta}\alpha\left(\tau_2+\phi_1+{\nu}\right)t}{M(\alpha)-(1-\alpha)\left(\tau_2+\phi_1+{\nu}\right)}\right]\\
&I_c(t)\geqslant I_c(0)Exp\left[\frac{-c^{1-\beta}\alpha\left(\tau_4+\phi_2+{\nu}\right)t}{M(\alpha)-(1-\alpha)\left(\tau_4+\phi_2+{\nu}\right)}\right],\\
&R(t)\geqslant R(0)Exp\left[-\frac{c^{1-\beta}\alpha\left(\tau{R}+{\nu}\right)t}{M(\alpha)-(1-\alpha)\left({\nu}\right)}\right],\\
&D(t)\geqslant D(0)Exp\left[-\frac{c^{1-\beta}\alpha\left({\nu}\right)t}{M(\alpha)-(1-\alpha)\left({\nu}\right)}\right],
\end{align*}
for all $t\geqslant{0}$, and for a Mittag-Leffler kernel we get
\begin{align*}
& S_p(t)\geqslant S_p(0)E_\alpha\left[\frac{-c^{1-\beta}\alpha\left(\frac{{\beta}\|I\|_\infty}{\|N\|_\infty}+{\sigma}+{\nu}\right)t^\alpha}{AB(\alpha)-(1-\alpha)\left(\frac{{\beta}\|I\|_\infty}{\|N\|_\infty}+{\sigma}+{\nu}\right)}\right],\\
& I(t)\geqslant I(0)E_\alpha\left[\frac{-c^{1-\beta}\alpha\left(\gamma_1+\gamma_2+\gamma_3+\gamma_4+\tau_1+{\nu}\right)t^\alpha}{AB(\alpha)-(1-\alpha)\left(\gamma_1+\gamma_2+\gamma_3+\gamma_4+\tau_1+{\nu}\right)}\right],\\
& I_p(t)\geqslant I_p(0)E_\alpha\left[\frac{-c^{1-\beta}\alpha\left(\tau_3+{\nu}\right)t^\alpha}{AB(\alpha)-(1-\alpha)\left(\tau_3+{\nu}\right)}\right],\\
& I_n(t)\geqslant I_n(0)E_\alpha\left[\frac{-c^{1-\beta}\alpha\left(\tau_2+\phi_1+{\nu}\right)t^\alpha}{AB(\alpha)-(1-\alpha)\left(\tau_2+\phi_1+{\nu}\right)}\right],\\
& I_c(t)\geqslant I_c(0)E_\alpha\left[\frac{-c^{1-\beta}\alpha\left(\tau_4+\phi_2+{\nu}\right)t^\alpha}{AB(\alpha)-(1-\alpha)\left(\tau_4+\phi_2+{\nu}\right)}\right],\\
& R(t)\geqslant R(0)E_\alpha\left[-\frac{c^{1-\beta}\alpha\left(\tau{R}+{\nu}\right)t^\alpha}{AB(\alpha)-(1-\alpha)\left({\nu}\right)}\right],\\
& D(t)\geqslant D(0)E_\alpha\left[-\frac{c^{1-\beta}\alpha\left({\nu}\right)t^\alpha}{AB(\alpha)-(1-\alpha)\left({\nu}\right)}\right] \text{ for all } t\ge 0.
\end{align*}

\section{Analysis of equilibrium points}\label{sec:equilibrium}
We have two equilibrium points in any of these system:
\begin{enumerate}
\item On the one hand, the disease-free equilibrium, with $I^*=0$, is given by
\begin{equation}\label{equilibfree}
	E_0^*=\left(\frac{\Pi}{\sigma+\nu},0,0,0,0,0,\frac{\sigma\Pi}{(\sigma+\nu)\nu}\right),
\end{equation}

\item On the other hand, the other equilibrium point  
$E_1^*=(S_p^*,I*,I_p^*,I_n^*,I_c^*,R^*,D^*)$ can be obtained by solving equations \eqref{eqn1}-\eqref{eqn7} with a null derivative, which gives
\begin{align*}
 S_p^*&=\frac{\lambda}{\beta},\\
 I^* &=\frac{\Pi}{\lambda}-\frac{({\sigma}+{\nu})}{\beta}, \\     I_p^*&=\frac{\gamma_3}{(\tau_3+{\nu})}\left(\frac{\Pi}{\lambda}-\frac{({\sigma}+{\nu})}{\beta}\right),\\ I_n^*&=\frac{\gamma_2}{(\tau_2+\phi_1+{\nu})}\left(\frac{\Pi}{\lambda}-\frac{({\sigma}+{\nu})}{\beta}\right),\\ I_c^*&=\frac{\gamma_3}{(\tau_4+\phi_2+{\nu})}\left(\frac{\Pi}{\lambda}-\frac{({\sigma}+{\nu})}{\beta}\right),\\ R^*&=\frac{1}{(\tau+\nu)}\left(\tau_1+\frac{\tau_2\gamma_2}{(\tau_2+\phi_1+{\nu})}+\frac{\tau_3\gamma_1}{(\tau_3+{\nu})}+\frac{\tau_4\gamma_3}{(\tau_4+\phi_2+{\nu})}\right)\left(\frac{\Pi}{\lambda}-\frac{({\sigma}+{\nu})}{\beta}\right),\\ D^*&=\frac{1}{\nu}\bigg[\sigma\left(\frac{\lambda}{\beta}\right)+\left(\frac{\Pi}{\lambda}-\frac{({\sigma}+{\nu})}{\beta}\right)\left(\frac{\phi_1\gamma_2}{(\tau_2+\phi_1+{\nu})}+\frac{\phi_2\gamma_3}{(\tau_4+\phi_2+{\nu})}+\gamma_4\right)\\ &+\frac{1}{(\tau+\nu)}\left(\tau_1+\frac{\tau_2\gamma_2}{(\tau_2+\phi_1+{\nu})}+\frac{\tau_3\gamma_1}{(\tau_3+{\nu})}+\frac{\tau_4\gamma_3}{(\tau_4+\phi_2+{\nu})}\right)\left(\frac{\Pi}{}-\frac{({\sigma}+{\nu})}{\beta}\right)\bigg]
\end{align*}
where $\lambda=\gamma_1+\gamma_2+\gamma_3+\gamma_4+\tau_1+\nu$.
\end{enumerate}

\section{Reproduction and strength numbers}\label{sec:rs_number}

Once computed the equilibrium points,
the reproduction number indicates the stability of a compartmental model at the equilibrium points. For computing this reproduction number, we use the next generation matrix method \cite{diekmann1990definition,van2002reproduction,van2017reproduction}
Let us consider equations \eqref{eqn2}-\eqref{eqn5}, involving $X=(I,I_p,I_n,I_c)$ for deriving the reproduction number. We can state them in a matrix form as $\frac{dX}{dt}=\mathcal{F}-\mathcal{V}$, where $\mathcal{F}$ gathers the rates of appearance of new infections in each compartment, and $\mathcal{V}$ shows the transition rates between compartments.\medskip 

Let us define the matrix $\mathcal{F}$ and $\mathcal{V}$ as 
$\mathcal{F}_{ij}=\frac{d\mathcal{F}_i(E_0^*)}{d X_j}$ and $V_{ij}=\frac{d\mathcal{V}_i(E_0^*)}{d X_j}$ 
at the disease-free equilibrium point $E_0^*$ shown in
\eqref{equilibfree} with $N=S_p+I+I_p+I_n+I_c+R+D$.
It can be seen that $
\mathcal{F}$ is non-null and $\mathcal{V}$ is a non-singular matrix \cite{berman1994nonnegative}.

\begin{equation}
\mathcal{F} =
\begin{pmatrix}
\frac{\beta{S_p}I}{N}   \\
0   \\
0 \\
0
\end{pmatrix},
\quad
\mathcal{V}=
\begin{pmatrix}
(\gamma_1+\gamma_2+\gamma_3+\gamma_4+\tau_1+{\nu})I   \\
-\gamma_1I-\phi_3I_c+(\tau_3+{\nu})I_p\\
-\gamma_2I-\phi_4I_c+(\tau_2+\phi_1+{\nu})I_n\\
-\gamma_3I+(\tau_4+\phi_2+{\nu})I_c
\end{pmatrix}
\end{equation}
where the entries of $\mathcal{F}$ and $\mathcal{V}$ can be obtained  $I,I_p,I_n,I_c$ at the free disease equilibrium point $E_0^*\left(\frac{\Pi}{\sigma+\nu},0,0,0,0,0,\frac{\sigma\Pi}{(\sigma+\nu)\nu}\right)$. This yields 
\begin{equation}
\mathcal{F}=
\begin{pmatrix}
\frac{\beta\nu}{\sigma+\nu} & 0 & 0 & 0   \\
0 & 0 & 0 & 0   \\
0 & 0 & 0 & 0 \\
0 & 0 & 0 & 0
\end{pmatrix},
\quad
\mathcal{V}=
\begin{pmatrix}
j_1 & 0 & 0 & 0   \\
-\gamma_1 & j_2 & 0 & \phi_3\\
-\gamma_2 & 0 & j_3 & 0\\
-\gamma_3 & 0 & 0 & j_4
\end{pmatrix}
\end{equation}
where $j_1=\gamma_1+\gamma_2+\gamma_3+\gamma_4+\tau_1+{\nu},j_2=\tau_3+{\nu},j_3=\tau_2+\phi_1+{\nu},j_4=\tau_4+\phi_2+{\nu}$.\medskip

Then the matrix $\mathcal{V}^{-1}$ will be
\begin{equation}
{\mathcal{V}^{-1}}=
\begin{pmatrix}
\frac{1}{j_1} & 0 & 0 & 0   \\
\frac{j_4\gamma_1}{j_1j_2j_4} & \frac{1}{j_2} & 0 & 0\\
\frac{j_4\gamma_2}{j_1j_3j_4} & 0 & \frac{1}{j_3} &0\\
\frac{\gamma_3}{j_1j_4} & 0 & 0 & \frac{1}{j_4}
\end{pmatrix}
\end{equation}

The matrix $FV^{-1}$ is called the next generation matrix. Its entry at position $(i,j)$ is the expected number of secondary infections of beings in compartment $i$ produced by an infected being from compartment $j$. Then, we get the reproduction number $\mathcal{R}_0=\rho(FV^{-1})$, where $\rho$ denotes the spectral radius of the matrix $FV^{-1}$, which yields 
\begin{equation}
	\mathcal{R}_0=\frac{\beta\nu}{(\sigma+\nu)j_1}
\end{equation}

Beyond the reproduction number $\mathcal{R}_0$, it is relevant to analyze in detail the matrix $\mathcal{F}_n$, which originates from the nonlinear segment of the infected classes.
\begin{align}
& \mathcal{F}_n =
\begin{pmatrix}
\frac{\partial^2 }{\partial I^2}\left(\frac{\beta{S_p}I}{N}\right)\bigg|_{E^0} & 0 & 0 & 0   \\
0 & 0 & 0 & 0   \\
0 & 0 & 0 & 0 \\
0 & 0 & 0 & 0
\end{pmatrix}
\quad
=
\begin{pmatrix}
\frac{-2\beta(\sigma+\nu)}{\Pi\left(1+\frac{\sigma}{\nu}\right)^2j_1} & 0 & 0 & 0   \\
0 & 0 & 0 & 0   \\
0 & 0 & 0 & 0 \\
0 & 0 & 0 & 0
\end{pmatrix},
\end{align}

This non-null term is known as the strength number $\mathcal{SN}$. It can be formally defined as $\mathcal{SN}=|F_nV^{-1}-\lambda{I}|=0$, which can be simplified to
\begin{equation}
	\mathcal{SN}=\frac{-2\beta(\sigma+\nu)}{\Pi\left(1+\frac{\sigma}{\nu}\right)^2j_1}
\end{equation}

If $\mathcal{SN}=0$, then it suggests that the spread might not undergo a review process, leading to a single magnitude before dissipating. A value of $\mathcal{SN}$ greater than zero indicates sufficient strength to initiate the renewal process, implying that the spread will exhibit multiple trends.  

\section{Stability analysis}
To describe the model's stability, we must first consider the system's Jacobian matrix at the model's equilibrium points. First, we investigate the stability at the free infection equilibrium point $E_0^*$.
Setting $A_1=\sigma+\nu$,  $A_2=\gamma_1+\gamma_2+\gamma_3+\gamma_4+\tau_1+\nu$, $A_3=\tau_3+\nu$, $A_4=\phi_1$, $A_5=\tau_4+\phi_2$, $A_6=\tau+\nu$, we get the Jacobian matrix at this point.
\begin{equation}
{J}=
\begin{pmatrix}
-A_1 & -\frac{\beta\nu}{\sigma+\nu} & 0 & 0 & 0 & 0 & 0 \\
0 & \frac{\beta\nu}{\sigma+\nu}-A_2 & 0 & 0 & 0 & 0 & 0\\
0 & \gamma_1 & -A_3 & 0 & 0 & 0 & 0\\
0 & \gamma_2 & 0 & -A_4 & 0 & 0 & 0\\
0 & \gamma_3 & 0 & 0 & -A_5 & 0 & 0\\
0 & \tau_1 & \tau_3 & \tau_2 & \tau_4 & A_6 & 0\\
\sigma & \gamma_4 & 0 & \phi_1 & \phi_2 & \tau & -\nu
\end{pmatrix}
\end{equation}

The eigenvalues calculated of $J(E_0^*)$ are $-\nu$, $A_1$, $\left(\frac{\beta\nu}{\sigma+\nu}-A_1\right)$, $-A_2$, $-A_3$, $-A_4$, $-A_5$. Except for $\lambda_3=\left(\frac{\beta\nu}{\sigma+\nu}-A_1\right)$, all of the eigenvalues have negative real parts, if it is rewrite as $\lambda_3=A_1\left(\frac{\beta\nu}{A_1(\sigma+\nu)}-1\right)$, which shows that if $\left(\frac{\beta\nu}{A_1(\sigma+\nu)}\right)<1$, then $\lambda_3$ can be less than zero, which coincides with the reproduction number $\mathcal{R}_0=\frac{\beta\nu}{A_1(\sigma+\nu)}$.\medskip

Now, let us analyze the behaviour around the endemic equilibrium point $E_1^*$. We define the endemic Lyapunov function setting all variables in the model to be at the harmful equilibrium $E_1^*$.

\begin{theorem}
	If the reproductive number $\mathcal{R}_0 > 1$, the endemic equilibrium point $E_1^*$ of (harmful information spreading) is globally asymptotically stable.
\end{theorem}
\begin{proof}
To begin the prove of the theorem, consider the Lyapunov function given by:
\begin{align}\label{eq:L}
 {L}&\nonumber=\left({S_p}-{S_p^*}-{S_p^*}\log\frac{S_p}{S_p^*}\right)+\left({I}-{I*}-{I*}\log\frac{I}{I*}\right),\\ &\nonumber +\left({I_p}-{I_p^*}-{I_p^*}\log\frac{I_p}{I_p^*}\right)+\left({I_n}-{I_n^*}-{I_n^*}\log\frac{I_n}{I_n^*}\right),\\ & +\left({I_c}-{I_c^*}-{I_c^*}\log\frac{I_c}{I_c^*}\right)+\left({R}-{R^*}-{R^*}\log\frac{R}{R^*}\right)+\left({D}-{D*}-{D^*}\log\frac{D}{D^*}\right),
\end{align}

Differentiating $L$ with respect to $t$, 
the following is obtained

\begin{align}\label{eq:dL}
\frac{dL}{dt}&\nonumber=\left(1-\frac{S_p^*}{S_p}\right)\dot{S_p}+\left(1-\frac{I^*}{I}\right)\dot{I} +\left(1-\frac{I_p^*}{I_p}\right)\dot{I_p}+\left(1-\frac{I_n^*}{I_n}\right)\dot{I_n}\\ & +\left(1-\frac{I_c^*}{I_c}\right)\dot{I_c}+\left(1-\frac{R^*}{R}\right)\dot{R}+\left(1-\frac{D^*}{D}\right)\dot{D},
\end{align}
by replacing the derivatives as in \eqref{eqn1}-\eqref{eqn7}, the following can be obtained

\begin{align}
\frac{dL}{dt}&\nonumber=\left(1-\frac{S_p^*}{S_p}\right)\left(\Pi-\frac{{\beta}S_pI}{N}-(\sigma+{\nu})S_p\right)\\ &\nonumber+\left(1-\frac{I^*}{I}\right)\left(\frac{{\beta}S_pI}{N}-(\gamma_1+\gamma_2+\gamma_3+\gamma_4+\tau_1+{\nu})\right)\\ &\nonumber +\left(1-\frac{I_p^*}{I_p}\right)\left(\gamma_1I-(\tau_3+-{\nu})I_p\right)\\ &\nonumber+\left(1-\frac{I_n^*}{I_n}\right)\left(\gamma_2I-(\tau_2+\phi_1+{\nu})I_n\right)\\ &\nonumber +\left(1-\frac{I_c^*}{I_c}\right)\left(\gamma_3I-(\tau_4+\phi_2+{\nu})I_c\right)\\ &\nonumber +\left(1-\frac{R^*}{R}\right)\left(\tau_1I+\tau_2I_n+\tau_3I_p+\tau_4I_c-(\tau+{\nu})R\right)\\ & +\left(1-\frac{D^*}{D}\right)\left({\sigma}S_p+\phi_1I_n+\phi_2I_c+\gamma_4{I}+\tau{R}-{\nu}D\right),
\end{align}
substitute $S_p=S_p-S_p^*$, $I_n=I_n-I_n^*$, $I=I-I^*$, $I_p=I_p-I_p^*$, $I_c=I_c-I_c^*$, $R=R-R^*$, $D=D-D^*$ 

\begin{align}
\frac{dL}{dt}&=\Omega - \Sigma
\end{align}
where
\begin{align}
	\Omega &\nonumber=\Pi+\frac{S_p^*}{S_p}\frac{{\beta}(S_p-S_p^*)(I-I^*)}{N}+\frac{S_p^*}{S_p}(\sigma+{\nu})(S_p-S_p^*)+\frac{{\beta}(S_p-S_p^*)(I-I^*)}{N}\\ &\nonumber +\frac{I^*}{I}(\gamma_1+\gamma_2+\gamma_3+\gamma_4+\tau_1+{\nu})(I-I^*)+\gamma_1(I-I^*)\\ &\nonumber +\frac{I_p^*}{I_p}(\tau_3+{\nu})(I_p-I_p^*)+\gamma_2(I-I^*) +\frac{I_n^*}{I_n}(\tau_2+\phi_1+{\nu})(I_n-I_n^*)+\gamma_3(I-I^*)\\ &\nonumber +\frac{I_c^*}{I_c}(\tau_4+\phi_2+{\nu})(I_c-I_c^*)+\tau_1(I-I^*)+\tau_2(I_n-I_n^*)+\tau_3(I_p-I_p^*)+\tau_4(I_c-I_c^*)\\ &\nonumber  +\frac{R^*}{R}(\tau+{\nu})(R-R^*)+{\sigma}(S_p-S_p^*)+\phi_1(I_n-I_n^*)+\phi_2(I_c-I_c^*)+\gamma_4{(I-I^*)}+\tau(R-R^*)\\ & +\frac{D^*}{D}{\nu}(D-D^*)
\end{align}
and
\begin{align}
	\Sigma &\nonumber = \left(\frac{S_p^*}{S_p}\Pi+\frac{{\beta}(S_p-S_p^*)(I-I^*)}{N}+(\sigma+{\nu})(S_p-S_p^*)\right)\\ &\nonumber +\frac{I^*}{I}\left(\frac{{\beta}(S_p-S_p^*)(I-I^*)}{N}+(\gamma_1+\gamma_2+\gamma_3+\gamma_4+\tau_1+{\nu})(I-I^*)\right)\\ &\nonumber +\frac{I_p^*}{I_p}\left(\gamma_1(I-I^*)(I_c-I_c^*)\right)-(\tau_3+{\nu})(I_p-I_p^*)+\frac{I_n^*}{I_n}\left(\gamma_2(I-I^*)\right)\\ &\nonumber +(\tau_2+\phi_1+{\nu})(I_n-I_n^*)+\frac{I_c^*}{I_c}\gamma_3(I-I^*)+(\tau_4+\phi_2+{\nu})(I_c-I_c^*)\\ &\nonumber  +\frac{R^*}{R}\left(\tau_1(I-I^*)+\tau_2(I_n-I_n^*)+\tau_3(I_p-I_p^*)+\tau_4(I_c-I_c^*)\right)\\ &\nonumber +\frac{D^*}{D}\left({\sigma}(S_p-S_p^*)+\phi_1(I_n-I_n^*)+\phi_2(I_c-I_c^*)+\gamma_4{(I-I^*)}+\tau(R-R^*)\right)\\ & +(\tau+{\nu})(R-R^*)-{\nu}(D-D^*).
\end{align}

Clearly, at the equilibrium point we get $\frac{dL}{dt}=0$ as expected. If $\Omega<\Sigma$ then $\frac{dL}{dt}<0$ and the endemic equilibrium $E_1^*$ is asymptotically locally stable. Besides, from Lasalle's invariance notion, it can be deduced that $E_1^*$ achieves global asymptotic stability under the condition $\Omega < \Sigma$ on the region $\Gamma$ defined as

\begin{equation}
\Gamma=\left\{(S_p^*,I^*,I_p^*,I_n^*,I_c^*,R^*,D^*)\in \Gamma\,\,\, \text{such that}\,\,\,\frac{dL}{dt}=0\right\}.
\end{equation}

This often lacks a clear condition regarding the sign of the Lyapunov function's first derivative.  
However, an alternative method to establish a clear condition involves utilizing the equilibrium points, which will be demonstrated below to distinguish between the two techniques. The steady-state results of the model are as follows:

\begin{equation}
\begin{cases}
\Pi=\nu{S_p^*}+\frac{\beta{S_p^*I^*}}{N^*}, A_1I^*,\\
A_1I^*=\frac{\beta{S_p^*I^*}}{N^*},A_2I^*_p,\\
A_2I^*_p=\gamma_1I^*,A_3I_n^*,\\
A_3I_n^*=\gamma_2I^*,A_4I_c^*,\\
A_4I_c^*=\gamma_3I^*,
\end{cases}
\end{equation}
\end{proof}

\begin{theorem}
	If $R_0\geqslant{1}$ and
	$$\left(5-\frac{S_pN^*}{NS_p^*}-\frac{S_p^*}{S_p}+\frac{2I}{I^*}-\frac{I_p}{I_p^*}-\frac{I_n}{I_n^*}-\frac{I_c}{I_c^*}-\frac{II_p^*}{I^*I_p^*}-\frac{II_n^*}{I^*I_n^*}-\frac{II_c^*}{I^*I_c^*}+\frac{IN^*}{NI^*}\right)\leqslant{0},$$
then endemic equilibrium is globally asymptotically stable.
\end{theorem}
\begin{proof} Let us consider the following nonlinear Lyapunov function for the proposed model. We recall that for the existence of the associated unique endemic equilibrium point E, we have $R_0 > 1$. 

\begin{align}
\mathcal{L}(t)&\nonumber=\int_{S_p^*}^{S_p}{\left(1-\frac{S_p^*}{x}\right)dx}+\int_{I^*}^{I}{\left(1-\frac{I^*}{x}\right)dx}\\ & +\frac{A_1}{\gamma_1}\int_{I_p^*}^{I_p}{\left(1-\frac{I_p^*}{x}\right)dx}+\frac{A_1}{\gamma_2}\int_{I_n^*}^{I_n}{\left(1-\frac{I_n^*}{x}\right)dx}+\frac{A_1}{\gamma_3}\int_{I_c^*}^{I_c}{\left(1-\frac{I_c^*}{x}\right)dx},
\end{align}
The derivative with respect to time of the above is

\begin{align}
\mathcal{L}^\prime(t)&\nonumber={\left(1-\frac{S_p^*}{S_p}\right)S_p^{\prime}(t)}+{\left(1-\frac{I^*}{I}\right)I^{\prime}(t)}\\ &\nonumber +\frac{A_1}{\gamma_1}{\left(1-\frac{I_p^*}{I_p}\right)I_p^{\prime}(t)}+\frac{A_1}{\gamma_2}{\left(1-\frac{I_n^*}{x}\right)I_n^{\prime}(t)}+\frac{A_1}{\gamma_3}{\left(1-\frac{I_c^*}{x}\right)I_c^{\prime}(t)},\\ &\nonumber =\left(1-\frac{S_p^*}{S_p}\right)\left(\Pi-\frac{{\beta}S_pI}{N}-{\nu}S_p\right)+\left(1-\frac{I^*}{I}\right)\left(\frac{{\beta}S_pI}{N}-(\gamma_1+\gamma_2+\gamma_3+\gamma_4+\tau_1+{\nu})I\right)\\ &\nonumber +\frac{A_1}{\gamma_1}\left(1-\frac{I_p^*}{I_p}\right)\left(\gamma_1I-(\tau_3+-{\nu})I_p\right)+\frac{A_1}{\gamma_2}\left(1-\frac{I_n^*}{I_n}\right)\left(\gamma_2I-(\tau_2+\phi_1+{\nu})I_n\right)\\ & +\frac{A_1}{\gamma_3}\left(1-\frac{I_c^*}{I_c}\right)\left(\gamma_3I-(\tau_4+\phi_2+{\nu})I_c\right),
\end{align}
by substituting the solutions one can have the following

After substituting the derivatives and simplifying, we get
\begin{align}
\mathcal{L}^\prime(t)&\nonumber = \frac{{\beta}S_p^*I^*}{N^*}\left(5-\frac{S_p^*}{S_p}-\frac{N^*S_p}{NS_p^*}-\frac{I_p^*I}{I^*I_P}-\frac{I_p}{I_p^*}-\frac{I_n^*I}{I^*I_n}-\frac{I_n}{I_n^*}-\frac{I_c^*I}{I^*I_c}-\frac{I_c}{I_c^*}+\frac{N^*I}{NI^*}\right)\\ & +\nu{S_p^*}\left(2-\frac{S_p}{S_p^*}-\frac{S_p^*}{S_p}\right)
\end{align}

Since the arithmetic mean is above the geometric mean, one gets the following estimation
\begin{align}
& \left(2-\frac{S_p}{S_p^*}-\frac{S_p^*}{S_p}\right)=-\frac{(S_p^*-S_p)^2}{S_p^*S_p}\leqslant{0},
\end{align}

In addition, if the following inequality is true,
\begin{equation}
\left(5-\frac{S_pN^*}{NS_p^*}-\frac{S_p^*}{S_p}+\frac{2I}{I^*}-\frac{I_p}{I_p^*}-\frac{I_n}{I_n^*}-\frac{I_c}{I_c^*}-\frac{II_p^*}{I^*I_p^*}-\frac{II_n^*}{I^*I_n^*}-\frac{II_c^*}{I^*I_c^*}+\frac{IN^*}{NI^*}\right)\leqslant{0}
\end{equation}
then $\mathcal{L}^\prime(t)\leqslant{0}$, for $R_0>1$. Thus, $\mathcal{L}(t)$ is a Lyapunov function in $\Omega$. Therefore, by Lasalle’s in variance principle, the following can be written as
\begin{align}
&\nonumber\lim\limits_{t\rightarrow{\infty}}S_p=S_p^*\,\,\, \lim\limits_{t\rightarrow{\infty}}I=I^*\,\,\,\lim\limits_{t\rightarrow{\infty}}I_p=I_p^*\\  &\lim\limits_{t\rightarrow{\infty}}I_n=I_n^*\,\,\, \lim\limits_{t\rightarrow{\infty}}I_c=I_c^*\,\,\,\lim\limits_{t\rightarrow{\infty}}R=R^*\,\,\,\, \,\,\,\lim\limits_{t\rightarrow{\infty}}D=D^*
\end{align}

So as to, as $t$ progresses towards infinity, every solution of the model converges toward its respective unique endemic equilibrium, corresponding to the relevant reproduction number. Specifically, when $R_0 > 1$, the endemic equilibrium point $E^*_1$ achieves global asymptotic stability.
\end{proof}

\subsection{Second derivative of Lyapunov}

Computing the second derivative function of the Lyapunov function \eqref{eq:L}, we can express it as
\begin{align}
\frac{d^2L}{dt^2}&\nonumber=\left(\frac{\dot{S_p}}{S_p}\right)^2{S_p^*} +\left(1-\frac{I^*}{I}\right)\left(\beta\left(\frac{(\dot{S_p}I+\dot{I}S_p)N-\dot{N}S_pI}{N^2}\right)\right)+\left(\frac{\dot{I}}{I}\right)^2{I^*} \\ &\nonumber+\left(1-\frac{I_p^*}{I_p}\right)(\gamma_1\dot{I})+\left(\frac{\dot{I_p}}{I_p}\right)^2{I_p^*}+\left(1-\frac{I_n^*}{I_n}\right)\left(\gamma_2\dot{I}\right)+\left(\frac{\dot{I_n}}{I_n}\right)^2{I_n^*}\\ &\nonumber +\left(1-\frac{I_c^*}{I_c}\right)\left(\gamma_3\dot{I}\right)+\left(\frac{\dot{I_c}}{I_c}\right)^2{I_c^*} +\left(\frac{\dot{R}}{R}\right)^2{R^*} +\left(1-\frac{R^*}{R}\right)\left(\tau_1\dot{I}+\tau_2\dot{I_n}+\tau_3\dot{I_p}+\tau_4\dot{I_c}\right)\\ &\nonumber +\left(\frac{\dot{D}}{D}\right)^2{D^*} +\left(1-\frac{D^*}{D}\right)\left(\sigma\dot{S_p}+\phi_1\dot{I_n}+\phi_2\dot{I_c}+\gamma_4\dot{I}+\tau\dot{R}\right)\\ &\nonumber -\bigg[
\left(1-\frac{S_p^*}{S_p}\right)\left((\sigma+\nu)\dot{S_p}+\beta\left(\frac{(\dot{S_p}I+\dot{I}S_p)N-\dot{N}S_pI}{N^2}\right)\right) +\left(1-\frac{I^*}{I}\right)\left(j_1\dot{I}\right)\\ &\nonumber +\left(1-\frac{I_p^*}{I_p}\right)(j_2I_p)+\left(1-\frac{I_n^*}{I_n}\right)\left(j_3\dot{I_n}\right)+\left(1-\frac{I_c^*}{I_c}\right)\left(j_4\dot{I_c}\right) +\left(1-\frac{R^*}{R}\right)\left(j_5\dot{R}\right)\\ & +\left(1-\frac{D^*}{D}\right)\left(\nu\dot{D}\right)\bigg],
\end{align}
and substituting \eqref{eqn1}-\eqref{eqn7}, it is possible to write as if there are separate components with a positive sign, $\Sigma_1$, and the components with a negative sign, $\Omega_1$, and express the sign of $\frac{d^2L}{dt^2}$ in terms of them.

\section{Equilibrium points for second order}

The second derivative provides valuable insights into curvatures within systems. Equilibrium points of the second-order derivatives of the solutions provides immensely valuable information in terms of stability. In particular, they are useful to study equilibrium points in dynamical systems when dealing with fractional operators that have memory effects. For instance, if the second derivative is positive, this yields stability.

\begin{align}
\frac{d^2 S_p}{dt^2}&\nonumber=-(\sigma+\nu)\dot{S_p}-\beta\left(\frac{(\dot{S_p}I+\dot{I}S_p)N-\dot{N}S_pI}{N^2}\right),\\  
\frac{d^2 I}{dt^2}&\nonumber=\beta\left(\frac{(\dot{S_p}I+\dot{I}S_p)N-\dot{N}S_pI}{N^2}\right)-j_1\dot{I}\\  \frac{d^2 I_p}{dt^2}&\nonumber=\gamma_1\dot{I}-j_2I_p,\\  \frac{d^2 I_n}{dt^2}&\nonumber=\gamma_2\dot{I}-j_3\dot{I_n},\\ \frac{d^2 I_c}{dt^2} &\nonumber=\gamma_3\dot{I}-j_4\dot{I_c},\\  \frac{d^2 R}{dt^2}&\nonumber=\tau_1\dot{I}+\tau_2\dot{I_n}+\tau_3\dot{I_p}+\tau_4\dot{I_c}-j_5\dot{R},\\  \frac{d^2 D}{dt^2}&=\sigma\dot{S_p}+\phi_1\dot{I_n}+\phi_2\dot{I_c}+\gamma_4\dot{I}+\tau\dot{R}-\nu\dot{D},
\end{align}

The following solution was obtained at the disease-free equilibrium point $E_0^*=\left(\frac{\Pi}{\sigma+\nu},0,0,0,0,0,\frac{\sigma\Pi}{(\sigma+\nu)\nu}\right)$.\medskip

\begin{align}
\frac{d S_p}{dt}&\nonumber=-\frac{j_1}{\sigma+\nu}\frac{d I}{dt},\\
\frac{d I_p}{dt}&\nonumber=\frac{\gamma_1}{j_2}\frac{d I}{dt},\\
\frac{d I_n}{dt}&\nonumber=\frac{\gamma_2}{j_3}\frac{d I}{dt},\\
\frac{d I_c}{dt}&\nonumber=\frac{\gamma_3}{j_4}\frac{d I}{dt},\\
\frac{d R}{dt}&\nonumber=\frac{1}{(\tau+\nu)}\left(\tau_1+\frac{\tau_2\gamma_2}{j_4}+\frac{\tau_3\gamma_1}{j_3}+\frac{\tau_4\gamma_3}{j_5}\right)\frac{d I}{dt}\\ \frac{d D}{dt}&=\frac{1}{\nu}\bigg[\sigma\dot{S_p}+\left(\frac{\phi_1\gamma_2}{j_4}+\frac{\phi_2\gamma_3}{j_5}+\gamma_4\right)\frac{d I}{dt}+\tau\frac{d R}{dt}\bigg]\\ 
&\nonumber
\end{align}
\begin{align}
\left(\Pi-\frac{{\beta}S_p^{**}I^{**}}{N^{**}}-({\sigma}+{\nu})S_p^{**}\right)&\nonumber=-\frac{j_1}{\sigma+\nu}\left(\frac{{\beta}S_p^{**}I^{**}}{N^{**}}-j_1I^{**}\right),\\
\left(\gamma_1I^{**}-j_2I_p^{**}\right)&\nonumber=\frac{\gamma_1}{j_2}\left(\frac{{\beta}S_p^{**}I^{**}}{N^{**}}-j_1I^{**}\right),\\
\left(\gamma_2I^{**}-j_3I_n^{**}\right)&\nonumber=\frac{\gamma_2}{j_3}\left(\frac{{\beta}S_p^{**}I^{**}}{N^{**}}-j_1I^{**}\right),\\
\left(\gamma_3I^{**}-j_4I_c^{**}\right)&\nonumber=\frac{\gamma_3}{j_4}\left(\frac{{\beta}S_p^{**}I^{**}}{N^{**}}-j_1I^{**}\right),\\
\tau_1I^{**}+\tau_2I_n^{**}+\tau_3I_p^{**}+\tau_4I_c^{**}-j_5R^{**}&\nonumber=\frac{1}{(\tau+\nu)}\left(\tau_1+\frac{\tau_2\gamma_2}{j_4}+\frac{\tau_3\gamma_1}{j_3}+\frac{\tau_4\gamma_3}{j_5}\right)\\ &\nonumber\times\left(\frac{{\beta}S_p^{**}I^{**}}{N^{**}}-j_1I^{**}\right)\\ {\sigma}S_p^{**}+\phi_1I_n^{**}+\phi_2I_c^{**}+\gamma_4{I^{**}}+\tau{R^{**}}-{\nu}D^{**}&\nonumber=\frac{1}{\nu}\bigg[\sigma\left(\Pi-\frac{{\beta}S_p^{**}I^{**}}{N^{**}}-({\sigma}+{\nu})S_p^{**}\right)\\ &\nonumber +\left(\frac{\phi_1\gamma_2}{j_4}+\frac{\phi_2\gamma_3}{j_5}+\gamma_4\right)\left(\frac{{\beta}S_p^{**}I^{**}}{N^{**}}-j_1I^{**}\right)\\ & +\tau\tau_1I^{**}+\tau_2I_n^{**}+\tau_3I_p^{**}+\tau_4I_c^{**}-j_5R^{**}\bigg]
\end{align}

While equilibrium points have been effectively utilized to glean insights into the spread of various diseases, it's crucial to acknowledge that these points are derived solely from the first derivative. This derivative, naturally, offers only an indication of the extent of change. On the other hand, $\,\,S_p^{**},I^{**},I_p^{**},I_n^{**},I_c^{**},R^{**}\,\,$ and $\,\,D^{**}\,\,$ can provide some relevant data concerning infection points that might be used to identify waves. As a result, if we do have two equilibrium points, people can consider two waves.\\

Now, to check the condition under which the following $\,\,I^{**}<0,\,\,\,I^{**}>0\,\,\,I^{**}=0\,\,$ holds, have to evaluate the sign test for $\,\,I^{**},\,\,\,I_p^{**},\,\,\,I_n^{**}\,\,\,I_c^{**}\,\,$.\\
\begin{equation}
I^{**}<0\implies \beta\left(\frac{(\dot{S_p}I+\dot{I}S_p)N-\dot{N}S_pI}{N^2}\right)-j_1\dot{I}<0,
\end{equation}
substituting the values of $\,\,\dot{S_p}\,\,\,\dot{I}\,\,$ and $\,\,\dot{N}=\Pi-\nu{N}\,\,$ one can have
\begin{align}
 \left(\frac{\beta\dot{S_p}I}{N}+\frac{\beta\dot{I}S_p}{N}-\frac{\beta\dot{N}S_pI}{N^2}\right)-j_1\beta\left({\frac{{\beta}S_pI}{N}-j_1I}\right)&\nonumber<0,\\  \nonumber\left(\beta\left({\Pi-\frac{{\beta}S_pI}{N}-({\sigma}+{\nu})S_p}\right)\frac{I}{N}+\beta\left({\frac{{\beta}S_pI}{N}-j_1I}\right)\frac{S_p}{N}-\frac{\beta(\Pi-\nu{N})S_pI}{N^2}\right)\\ -j_1\beta\left({\frac{{\beta}S_pI}{N}-j_1I}\right)&<0,
\end{align}
by multiplying both side by $N$ and recall that $I\ne{0}$ one can have the following after simplification
\begin{align}
&\beta\Pi+j_1^2N+\nu\beta{S_p}<\beta(\sigma+\nu)S_p+2\beta{j_1}S_p+\frac{\beta^2S_pI}{N}+\frac{\beta\Pi{S_p}}{N}
\end{align}
Setting $N>S_p$ and $N>I$, use in the above one can have the following
\begin{align}
&j_1^2N<\beta\sigma{S_p}+2j_1S_p+{\beta^2S_p}\implies S_p>\frac{j_1^2}{\beta\sigma+2\beta{j_1}+{\beta^2}}
\end{align}
When these criteria are fulfilled, $I^{**}<0$ signifies the presence of a local maximum in the function. Conversely, if $I^{**}> 0$, it indicates the existence of a local minimum. When $I^{**}=0$, it suggests a potential inflection point within the function. However, we express these conditions in the following manner:
\begin{align}
&\nonumber \frac{d^2 I_p}{dt^2}<0\implies \gamma_1\dot{I}-j_2\dot{I_p}<0,\\
&\nonumber \gamma_1\left(\frac{{\beta}S_pI}{N}-j_1I\right)-j_2(\gamma_1I-j_2I_p)<0\\
&\frac{\gamma_1{\beta}S_pI}{N}-\gamma_1j_1I-j_2\gamma_1I+j_2^2I_p<0,
\end{align}
using $I>I_p$ given by the definition one can write as
 \begin{align}
 &\nonumber\frac{\gamma_1{\beta}S_pI}{N}+j_2^2I_p<\gamma_1(j_1+j_2)I,\\
 &\nonumber\frac{\gamma_1{\beta}S_pI}{N}+j_2^2I<\gamma_1(j_1+j_2)I+j_2^2I,\\
 &\frac{\gamma_1{\beta}S_p}{N}<\gamma_1(j_1+j_2)\implies S_p<\frac{(j_1+j_2)N}{{\beta}},
 \end{align}
\begin{align}
&\nonumber \frac{d^2 I_n}{dt^2}<0\implies \gamma_2\dot{I}-j_3\dot{I_n}<0,\\
&\nonumber \gamma_2\left(\frac{{\beta}S_pI}{N}-j_1I\right)-j_3(\gamma_2I-j_3I_n)<0\\
&\nonumber\frac{\gamma_2{\beta}S_pI}{N}-\gamma_2j_1I-j_3\gamma_2I+j_3^2I_n<0,\\
&\nonumber\frac{\gamma_2{\beta}S_pI}{N}+j_3^2I_n<\gamma_2(j_1+j_3)I,\\
&\nonumber\frac{\gamma_2{\beta}S_pI}{N}+j_3^2I<\gamma_2(j_1+j_3)I+j_3^2I,\\
&\frac{\gamma_2{\beta}S_p}{N}<\gamma_2(j_1+j_3)\implies S_p<\frac{(j_1+j_3)N}{{\beta}},
\end{align}
\begin{align}
&\nonumber \frac{d^2 I_c}{dt^2}<0\implies \gamma_3\dot{I}-j_4\dot{I_c}<0,\\
&\nonumber \gamma_3\left(\frac{{\beta}S_pI}{N}-j_1I\right)-j_4(\gamma_3I-j_4I_c)<0\\
&\nonumber\frac{\gamma_3{\beta}S_pI}{N}-\gamma_3j_1I-j_4\gamma_3I+j_4^2I_c<0,\\
&\nonumber\frac{\gamma_3{\beta}S_pI}{N}+j_4^2I_n<\gamma_3(j_1+j_4)I,\\
&\nonumber\frac{\gamma_3{\beta}S_pI}{N}+j_4^2I<\gamma_3(j_1+j_4)I+j_4^2I,\\
&\frac{\gamma_3{\beta}S_p}{N}<\gamma_3(j_1+j_4)\implies S_p<\frac{(j_1+j_4)N}{{\beta}},
\end{align}
Moreover, to attain the maximum value of time for the four classes, the following conditions must hold true:
\begin{equation}
	\frac{j_1^2}{\beta\sigma+2\beta{j_1}+{\beta^2}}<S_p<min\left\{\frac{(j_1+j_2)N}{{\beta}},\frac{(j_1+j_3)N}{{\beta}},\frac{(j_1+j_4)N}{{\beta}}\right\}
\end{equation}

\section{Numerical iterative scheme}
\label{sec:numerical}

In order to illustrate some results, we show the discretization and graphs of the solutions provided by the fractal fractional model with their three respetive kernels.\medskip

\subsection{Power law}
\begin{align}
	^{FFP}_0D^{\alpha,\beta}_tS_p(t)&=\Pi-\frac{\beta{S_p}{I}}{N}-(\sigma+\nu)S_p\\
	^{RL}_{0}D^{\alpha}_t{S_p(t)}&=\beta{t^{\beta-1}}S_p^{\prime}=\beta{t^{\beta-1}}\left[\Pi-\frac{\beta{S_p}{I}}{N}-(\sigma+\nu)S_p\right]
\end{align}

Let $F_1(t,S_p,I)=^{RL}_{0}D^{\alpha}_t{S_p(t)}=\beta{t^{\beta-1}}\left[\Pi-\frac{\beta{S_p}{I}}{N}-(\sigma+\nu)S_p\right]$ and applying the Riemann Liouville integral definition to get
\begin{align}
\nonumber S_p(t)&=S_p(0)+\frac{1}{\Gamma(\alpha)}\int_{0}^{t}F_1(\tau,S_p,I)(t-\tau)^{\alpha-1}d\tau \\ S_p(t_{n+1})&=S_p(0)+\frac{1}{\Gamma(\alpha)}\int_{0}^{t}F_1(\tau,S_p,I)(t-\tau)^{\alpha-1}d\tau\\ & \nonumber \text{and with of help discretization one can obtained the following equations} \\ &=S_p(0)+\frac{1}{\Gamma(\alpha)}\sum_{j=0}^{n}\int_{t_j}^{t_{j+1}}A_1(\tau,S_p,I)(t_{n+1}-\tau)^{\alpha-1}d\tau
\end{align}

Using two steps Lagrange interpolation polynomial approximations
\begin{equation}
P_j=\left(\frac{\tau-t_{j-1}}{t_j-t_{j-1}}\right)F_1(t_j,S_p,I)-\left(\frac{\tau-t_{j}}{t_j-t_{j-1}}\right)F_1(t_{j-1},S_p,I)
\end{equation}

\begin{align}
 \nonumber S_p(t_{n+1})&=S_p(0)+\frac{1}{\Gamma(\alpha)}\int_{0}^{t_{n+1}}F_1(\tau,S_p,I)(t_{n+1}-\tau)^{\alpha-1}d\tau\\ &\nonumber=S_p(0)+\frac{1}{\Gamma(\alpha)}\sum_{j=0}^{n}\int_{t_j}^{t_{j+1}}\bigg[\left(\frac{\tau-t_{j-1}}{t_j-t_{j-1}}\right)F_1(t_j,S_p,I)-\left(\frac{\tau-t_{j}}{t_j-t_{j-1}}\right)F_1(t_{j-1},S_p,I)\bigg](t_{n+1}-\tau)^{\alpha-1}d\tau
\end{align}
by evaluating the integrals and substitute $h=t_j-t_{j-1},\,\,t_j=jh$ and simplify to get have simplification obtained the following
\begin{align}
\nonumber S_p(t_{n+1})& =S_p(0)+\frac{h^\alpha}{\Gamma(\alpha+2)}\bigg[\sum_{j=0}^{n}F_1(t_j,S_p,I)\big[(-j+n+1)^\alpha(-j+n+\alpha+2)-(-j+n)^\alpha(-j+n+2\alpha+2)\big]\\ & -\big[\sum_{j=0}^{n}F_1(t_{j-1},S_p,I)\big[(-j+n+1)^{\alpha+1}-(-j+n)^\alpha(-j+n+\alpha+)\big]\bigg]
\end{align}
Following the same procedure to obtained the following
\begin{align}
\nonumber I(t_{n+1})& =I(0)+\frac{h^\alpha}{\Gamma(\alpha+2)}\bigg[\sum_{j=0}^{n}F_2(t_j,S_p,I)\big[(-j+n+1)^\alpha(-j+n+\alpha+2)-(-j+n)^\alpha(-j+n+2\alpha+2)\big]\\ & -\big[\sum_{j=0}^{n}F_2(t_{j-1},S_p,I)\big[(-j+n+1)^{\alpha+1}-(-j+n)^\alpha(-j+n+\alpha+)\big]\bigg]
\end{align}
\begin{align}
\nonumber I_p(t_{n+1})& =I_p(0)+\frac{h^\alpha}{\Gamma(\alpha+2)}\bigg[\sum_{j=0}^{n}F_3(t_j,I,I_p)\big[(-j+n+1)^\alpha(-j+n+\alpha+2)-(-j+n)^\alpha(-j+n+2\alpha+2)\big]\\ & -\big[\sum_{j=0}^{n}F_3(t_{j-1},I,I_p)\big[(-j+n+1)^{\alpha+1}-(-j+n)^\alpha(-j+n+\alpha+)\big]\bigg]
\end{align}
\begin{align}
\nonumber I_n(t_{n+1})& =I_n(0)+\frac{h^\alpha}{\Gamma(\alpha+2)}\bigg[\sum_{j=0}^{n}F_4(t_j,I,I_n)\big[(-j+n+1)^\alpha(-j+n+\alpha+2)-(-j+n)^\alpha(-j+n+2\alpha+2)\big]\\ & -\big[\sum_{j=0}^{n}F_4(t_{j-1},I,I_n)\big[(-j+n+1)^{\alpha+1}-(-j+n)^\alpha(-j+n+\alpha+)\big]\bigg]
\end{align}
\begin{align}
\nonumber I_c(t_{n+1})& =I_c(0)+\frac{h^\alpha}{\Gamma(\alpha+2)}\bigg[\sum_{j=0}^{n}F_5(t_j,I,I_c)\big[(-j+n+1)^\alpha(-j+n+\alpha+2)-(-j+n)^\alpha(-j+n+2\alpha+2)\big]\\ & -\big[\sum_{j=0}^{n}F_5(t_{j-1},I,I_c)\big[(-j+n+1)^{\alpha+1}-(-j+n)^\alpha(-j+n+\alpha+)\big]\bigg]
\end{align}
\begin{align}
\nonumber R(t_{n+1})& =R(0)+\frac{h^\alpha}{\Gamma(\alpha+2)}\bigg[\sum_{j=0}^{n}F_6(t_j,I,I_p,I_n,I_c,R)\big[(-j+n+1)^\alpha(-j+n+\alpha+2)-(-j+n)^\alpha(-j+n+2\alpha+2)\big]\\ & -\big[\sum_{j=0}^{n}F_6(t_{j-1},I,I_p,I_n,I_c,R)\big[(-j+n+1)^{\alpha+1}-(-j+n)^\alpha(-j+n+\alpha+)\big]\bigg]
\end{align}
\begin{align}
\nonumber D(t_{n+1})& =D(0)+\frac{h^\alpha}{\Gamma(\alpha+2)}\bigg[\sum_{j=0}^{n}A_4(t_j,S_p,I,I_n,I_c,R,D)\big[(-j+n+1)^\alpha(-j+n+\alpha+2)\\
& -(-j+n)^\alpha(-j+n+2\alpha+2)\big]\\ & -\big[\sum_{j=0}^{n}A_4(t_{j-1},S_p,I,I_n,I_c,R,D)\big[(-j+n+1)^{\alpha+1}-(-j+n)^\alpha(-j+n+\alpha+)\big]\bigg]
\end{align}
\subsection{Exponential kernel}
\begin{align}
^{FFE}_0D^{\alpha,\beta}_tS_p(t)&=\Pi-\frac{\beta{S_p}{I}}{N}-(\sigma+\nu)S_p\\
^{CF}_{0}D^{\alpha}_t{S_p(t)}&=\beta{t^{\beta-1}}S_p^{\prime}=\beta{t^{\beta-1}}\left[\Pi-\frac{\beta{S_p}{I}}{N}-(\sigma+\nu)S_p\right]
\end{align}
Let $M_1(t,S_p,I)=^{CF}_{0}D^{\alpha}_t{S_p(t)}=\beta{t^{\beta-1}}\left[\Pi-\frac{\beta{S_p}{I}}{N}-(\sigma+\nu)S_p\right]$
\begin{align}
\therefore S_p(t)&=S_p(0)+\frac{1-\alpha}{M(\alpha)}M_1(\tau,S_p,I)+\frac{\alpha}{M(\alpha)}\int_{0}^{t}M_1(\tau,S_p,I)d\tau\\ & \nonumber \text{and with of help discretization one can obtained the following equations} \\ S_p(t_{n+1}) &=S_p(0)+\frac{1-\alpha}{M(\alpha)}M_1(t_n,S_p^n,I^n)+\frac{\alpha}{M(\alpha)}\int_{0}^{t_{n+1}}M_1(\tau,S_p,I)d\tau\\ S_p(t_{n}) &=S_p(0)+\frac{1-\alpha}{M(\alpha)}M_1(t_n,S_p^{n-1},I^{n-1})+\frac{\alpha}{M(\alpha)}\int_{0}^{t_{n}}M_1(\tau,S_p,I)d\tau
\end{align}
taking the difference of the above last two equation to have the following
\begin{align}
& S_p(t_{n+1})-S_p(t_{n})=\frac{1-\alpha}{M(\alpha)}\left[M_1(t_n,S_p^n,I^n)-M_1(t_n,S_p^{n-1},I^{n-1})\right]+\frac{\alpha}{M(\alpha)}\int_{t_n}^{t_{n+1}}M_1(\tau,S_p,I)d\tau
\end{align}
Using two steps Lagrange polynomial Approximations
\begin{align}
\nonumber S_p(t_{n+1})&=S_p(t_{n})+\frac{1-\alpha}{M(\alpha)}\left[M_1(t_n,S_p^n,I^n)-M_1(t_n,S_p^{n-1},I^{n-1})\right]\\ &\nonumber +\frac{\alpha}{M(\alpha)}\bigg(M_1(t_n,S_p,I)\int_{t_n}^{t_{n+1}}\left(\frac{\tau-t_{n-1}}{t_n-t_{n-1}}\right)d\tau\\ &-M_1(t_{n-1},S_p,I)\int_{t_n}^{t_{n+1}}\left(\frac{\tau-t_{n}}{t_n-t_{n-1}}\right)d\tau\bigg)\\ &\nonumber \text{by evaluating the integrals and substitute $h=t_n-t_{n-1},\,\,t_n=nh$ and simplify to get have}\\ &\nonumber =S_p(t_{n})+\frac{1-\alpha}{M(\alpha)}\left[M_1(t_n,S_p^n,I^n)-M_1(t_n,S_p^{n-1},I^{n-1})\right]\\ &+\frac{\alpha}{M(\alpha)}\bigg(\frac{h(h+1)}{2}M_1(t_n,S_p,I)-\left(\frac{(2n+1)h^2}{2}+nh\right)M_1(t_{n-1},S_p,I)\bigg)
\end{align}
by adopting the same procedure one can obtained the following
\begin{align}
\nonumber I(t_{n+1}) &=I(t_{n})+\frac{1-\alpha}{M(\alpha)}\left[M_2(t_n,S_p^n,I^n)-M_2(t_n,S_p^{n-1},I^{n-1})\right]\\ &+\frac{\alpha}{M(\alpha)}\bigg(\frac{h(h+1)}{2}M_2(t_n,S_p,I)-\left(\frac{(2n+1)h^2}{2}+nh\right)M_2(t_{n-1},S_p,I)\bigg)
\end{align}
\begin{align}
\nonumber I_p(t_{n+1}) &=I_p(t_{n})+\frac{1-\alpha}{M(\alpha)}\left[M_3(t_n,I^n,I_p^n)-M_3(t_n,I^{n-1},I_p^{n-1})\right]\\ &+\frac{\alpha}{M(\alpha)}\bigg(\frac{h(h+1)}{2}M_3(t_n,I,I_p)-\left(\frac{(2n+1)h^2}{2}+nh\right)M_3(t_{n-1},I,I_p)\bigg)
\end{align}
\begin{align}
\nonumber I_n(t_{n+1}) &=I_n(t_{n})+\frac{1-\alpha}{M(\alpha)}\left[M_4(t_n,I^n,I_n^n)-M_4(t_n,I^{n-1},I_n^{n-1})\right]\\ &+\frac{\alpha}{M(\alpha)}\bigg(\frac{h(h+1)}{2}M_4(t_n,I,I_n)-\left(\frac{(2n+1)h^2}{2}+nh\right)M_4(t_{n-1},I,I_n)\bigg)
\end{align}
\begin{align}
\nonumber I_c(t_{n+1}) &=I_c(t_{n})+\frac{1-\alpha}{M(\alpha)}\left[M_5(t_n,I^n,I_c^n)-M_5(t_n,I^{n-1},I_c^{n-1})\right]\\ &+\frac{\alpha}{M(\alpha)}\bigg(\frac{h(h+1)}{2}M_5(t_n,S_p,I)-\left(\frac{(2n+1)h^2}{2}+nh\right)M_5(t_{n-1},S_p,I)\bigg)
\end{align}
\begin{align}
\nonumber R(t_{n+1}) &=R(t_{n})+\frac{1-\alpha}{M(\alpha)}\left[M_6(t_n,I^n,I_p^n,I_n^n,I_c^n,R^n)-M_6(t_n,I^{n-1},I_p^{n-1},I_n^{n-1},I_c^{n-1},R^{n-1})\right]\\ &+\frac{\alpha}{M(\alpha)}\bigg(\frac{h(h+1)}{2}M_6(t_n,I,I_p,I_n,I_c,R)-\left(\frac{(2n+1)h^2}{2}+nh\right)M_6(t_{n-1},I,I_p,I_n,I_c,R)\bigg)
\end{align}
\begin{align}
\nonumber D(t_{n+1}) &=D(t_{n})+\frac{1-\alpha}{M(\alpha)}\left[M_7(t_n,S_p^n,I^n,I_n^n,I_c^n,R^n,D^n)-M_7(t_n,S_p^{n-1},I^{n-1},I_n^{n-1},I_c^{n-1},R^{n-1},D^{n-1})\right]\\ &+\frac{\alpha}{M(\alpha)}\bigg(\frac{h(h+1)}{2}M_7(t_n,S_p,I,I_n,I_c,R,D)-\left(\frac{(2n+1)h^2}{2}+nh\right)M_7(t_{n-1},S_p,I,I_n,I_c,R,D)\bigg)
\end{align}
\subsection{Mittag-Leffler kernel}
\begin{align}
^{FFM}_0D^{\alpha,\beta}_tS_p(t)&=\Pi-\frac{\beta{S_p}{I}}{N}-(\sigma+\nu)S_p\\
^{AB}_{0}D^{\alpha}_t{S_p(t)}&=\beta{t^{\beta-1}}S_p^{\prime}=\beta{t^{\beta-1}}\left[\Pi-\frac{\beta{S_p}{I}}{N}-(\sigma+\nu)S_p\right]
\end{align}
Let $N_1(t,S_p,I)=^{AB}_{0}D^{\alpha}_t{S_p(t)}=\beta{t^{\beta-1}}\left[\Pi-\frac{\beta{S_p}{I}}{N}-(\sigma+\nu)S_p\right]$, and apply AB integral definition
\begin{align}
\therefore  \nonumber S_p(t)-S_p(0)&=\frac{1-\alpha}{AB(\alpha)}N_1(t,S_p,I)+\frac{\alpha}{AB(\alpha)\Gamma(\alpha)}\int_{0}^{t}N_1(\tau,S_p,I)(t-\tau)^{\alpha-1}d\tau\\ &\nonumber \text{by discretizing the above at $t_{n+1}$ to have the following}\\ S_p(t_{n+1}) &=S_p(0)+\frac{1-\alpha}{AB(\alpha)}N_1(\tau,S_p^n,I^n)+\frac{\alpha}{AB(\alpha)\Gamma(\alpha)}\int_{0}^{t_{n+1}}N_1(\tau,S_p,I)(t_{n+1}-\tau)^{\alpha-1}d\tau
\end{align}
within the interval $[t_n,t_{n+1}]$ and by using two steps Lagrange polynomial Approximations on the function $N_1(\tau,S_p,I)$ one can gets
\begin{align}
\nonumber S_p(t_{n+1})&=S_p(0)+\frac{1-\alpha}{AB(\alpha)}N_1(\tau,S_p^n,I^n)+\frac{\alpha}{AB(\alpha)\Gamma(\alpha)}\int_{0}^{t_{n+1}}N_1(\tau,S_p,I)(t_{n+1}-\tau)^{\alpha-1}d\tau\\ &\nonumber=S_p(0)+\frac{1-\alpha}{AB(\alpha)}N_1(\tau,S_p^n,I^n)+\frac{\alpha}{AB(\alpha)\Gamma(\alpha)}\sum_{j=0}^{n}\int_{t_j}^{t_{j+1}}\bigg[\left(\frac{\tau-t_{j-1}}{t_j-t_{j-1}}\right)N_1(t_j,S_p,I)\\ &-\left(\frac{\tau-t_{j}}{t_j-t_{j-1}}\right)N_1(t_{j-1},S_p,I)\bigg](t_{n+1}-\tau)^{\alpha-1}d\tau
\end{align}
by evaluating the integrals and substitute $h=t_j-t_{j-1},\,\,t_j=jh$ and simplify to get have simplification obtained the following
\begin{align}
\nonumber S_p(t_{n+1})& =S_p(0)+\frac{1-\alpha}{AB(\alpha)}N_1(\tau,S_p^n,I^n)\\ &\nonumber+\frac{{\alpha}h^\alpha}{AB(\alpha)\Gamma(\alpha+2)}\bigg[\sum_{j=0}^{n}N_1(t_j,S_p,I)\big[(-j+n+1)^\alpha(-j+n+\alpha+2)\\ &\nonumber-(-j+n)^\alpha(-j+n+2\alpha+2)\big] -\big[\sum_{j=0}^{n}N_1(t_{j-1},S_p,I)\big[(-j+n+1)^{\alpha+1}\\ &-(-j+n)^\alpha(-j+n+\alpha+2)\big]\bigg]
\end{align}
by adopting the same procedure one can obtained the following
\begin{align}
\nonumber I(t_{n+1})& =I(0)+\frac{1-\alpha}{AB(\alpha)}N_2(\tau,S_p^n,I^n)\\ &\nonumber+\frac{{\alpha}h^\alpha}{AB(\alpha)\Gamma(\alpha+2)}\bigg[\sum_{j=0}^{n}N_1(t_j,S_p,I)\big[(-j+n+1)^\alpha(-j+n+\alpha+2)\\ &\nonumber-(-j+n)^\alpha(-j+n+2\alpha+2)\big] -\big[\sum_{j=0}^{n}N_2(t_{j-1},S_p,I)\big[(-j+n+1)^{\alpha+1}\\ &-(-j+n)^\alpha(-j+n+\alpha+2)\big]\bigg]
\end{align}
\begin{align}
\nonumber I(t_{n+1})& =I(0)+\frac{1-\alpha}{AB(\alpha)}N_2(\tau,S_p^n,I^n)\\ &\nonumber+\frac{{\alpha}h^\alpha}{AB(\alpha)\Gamma(\alpha+2)}\bigg[\sum_{j=0}^{n}N_2(t_j,S_p,I)\big[(-j+n+1)^\alpha(-j+n+\alpha+2)\\ &\nonumber-(-j+n)^\alpha(-j+n+2\alpha+2)\big] -\big[\sum_{j=0}^{n}N_2(t_{j-1},S_p,I)\big[(-j+n+1)^{\alpha+1}\\ &-(-j+n)^\alpha(-j+n+\alpha+2)\big]\bigg]
\end{align}
\begin{align}
\nonumber I_n(t_{n+1})& =I_n(0)+\frac{1-\alpha}{AB(\alpha)}N_4(\tau,S_p^n,I^n)\\ &\nonumber+\frac{{\alpha}h^\alpha}{AB(\alpha)\Gamma(\alpha+2)}\bigg[\sum_{j=0}^{n}N_4(t_j,S_p,I)\big[(-j+n+1)^\alpha(-j+n+\alpha+2)\\ &\nonumber-(-j+n)^\alpha(-j+n+2\alpha+2)\big] -\big[\sum_{j=0}^{n}N_4(t_{j-1},S_p,I)\big[(-j+n+1)^{\alpha+1}\\ &-(-j+n)^\alpha(-j+n+\alpha+2)\big]\bigg]
\end{align}
\begin{align}
\nonumber I_c(t_{n+1})& =I_c(0)+\frac{1-\alpha}{AB(\alpha)}N_5(\tau,S_p^n,I^n)\\ &\nonumber+\frac{{\alpha}h^\alpha}{AB(\alpha)\Gamma(\alpha+2)}\bigg[\sum_{j=0}^{n}N_5(t_j,S_p,I)\big[(-j+n+1)^\alpha(-j+n+\alpha+2)\\ &\nonumber-(-j+n)^\alpha(-j+n+2\alpha+2)\big] -\big[\sum_{j=0}^{n}N_5(t_{j-1},S_p,I)\big[(-j+n+1)^{\alpha+1}\\ &-(-j+n)^\alpha(-j+n+\alpha+2)\big]\bigg]
\end{align}
\begin{align}
\nonumber R(t_{n+1})& =R(0)+\frac{1-\alpha}{AB(\alpha)}N_6(\tau,I^n,I_p^n,I_n^n,I_c^n,R^n)\\ &\nonumber+\frac{{\alpha}h^\alpha}{AB(\alpha)\Gamma(\alpha+2)}\bigg[\sum_{j=0}^{n}N_6(t_j,I,I_p,I_n,I_c,R)\big[(-j+n+1)^\alpha(-j+n+\alpha+2)\\ &\nonumber-(-j+n)^\alpha(-j+n+2\alpha+2)\big] -\big[\sum_{j=0}^{n}N_6(t_{j-1},I,I_p,I_n,I_c,R)\big[(-j+n+1)^{\alpha+1}\\ &-(-j+n)^\alpha(-j+n+\alpha+2)\big]\bigg]
\end{align}
\begin{align}
\nonumber D(t_{n+1})& =D(0)+\frac{1-\alpha}{AB(\alpha)}N_7(\tau,S_p^n,										I^n,I_n^n,I_c^n,R^n,D^n)\\ &\nonumber+\frac{{\alpha}h^\alpha}{AB(\alpha)\Gamma(\alpha+2)}\bigg[\sum_{j=0}^{n}N_7(t_j,S_p,I,I_n,I_c,R,D)\big[(-j+n+1)^\alpha(-j+n+\alpha+2)\\ &\nonumber-(-j+n)^\alpha(-j+n+2\alpha+2)\big] -\big[\sum_{j=0}^{n}N_7(t_{j-1},S_p,I,I_n,I_c,R,D)\big[(-j+n+1)^{\alpha+1}\\ &-(-j+n)^\alpha(-j+n+\alpha+2)\big]\bigg]
\end{align}

\subsection{Numerical simulations}
We illustrate our results with some plots. In Figures 1, we can see the effect of the power law. In Figures 8-14, we can see the effect of the exponential decay kernel. In Figures 15-21, we can see the effect of the Mittag-Leffler kernel.

\begin{figure}[H]
	\centering
	\includegraphics[width=6.5cm]{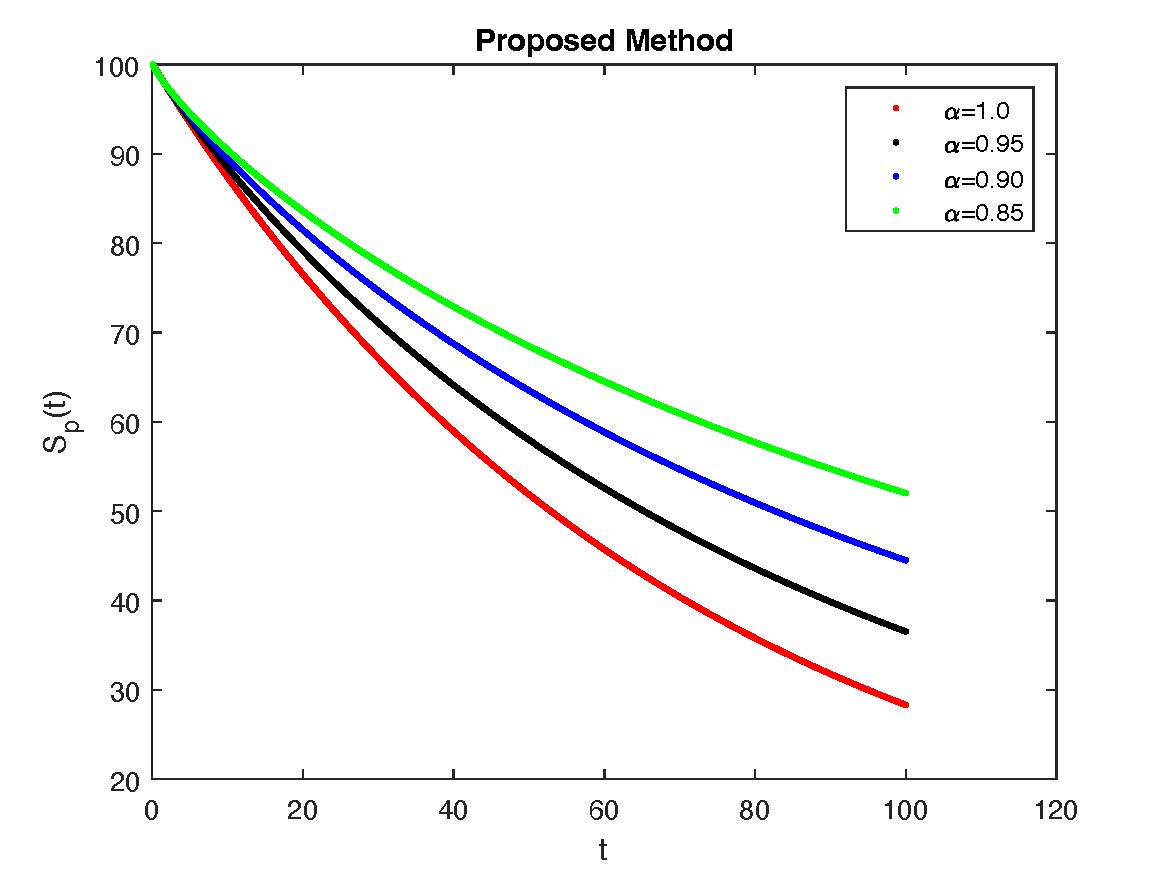}
	\includegraphics[width=6.5cm]{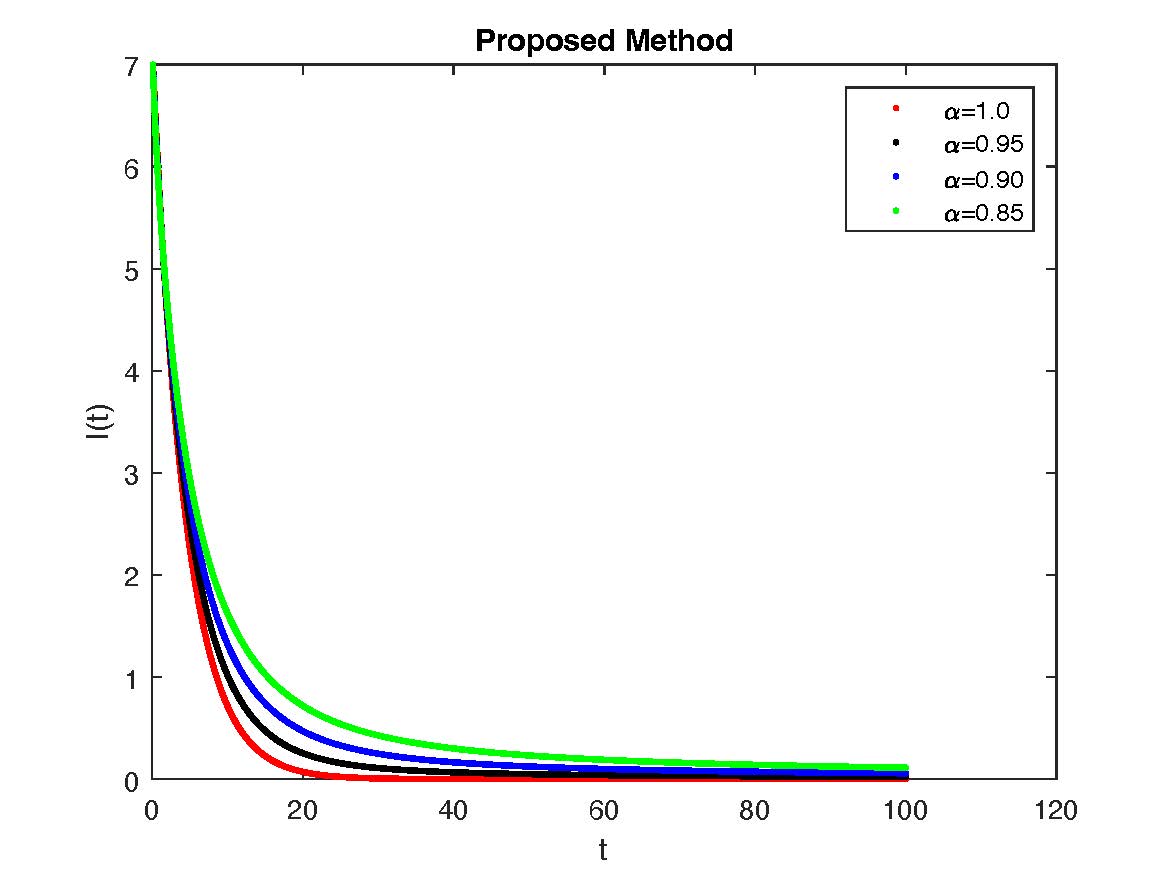}
	\includegraphics[width=6.5cm]{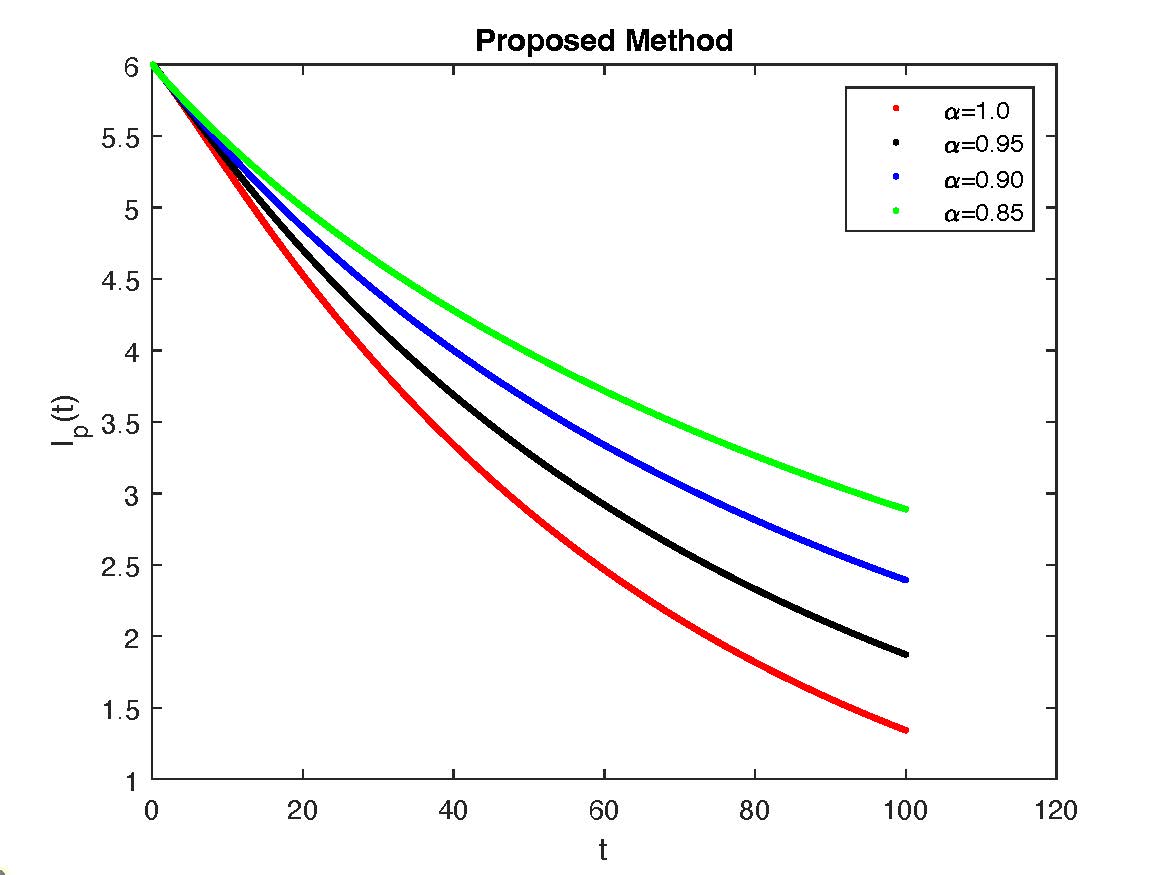}
	\includegraphics[width=6.5cm]{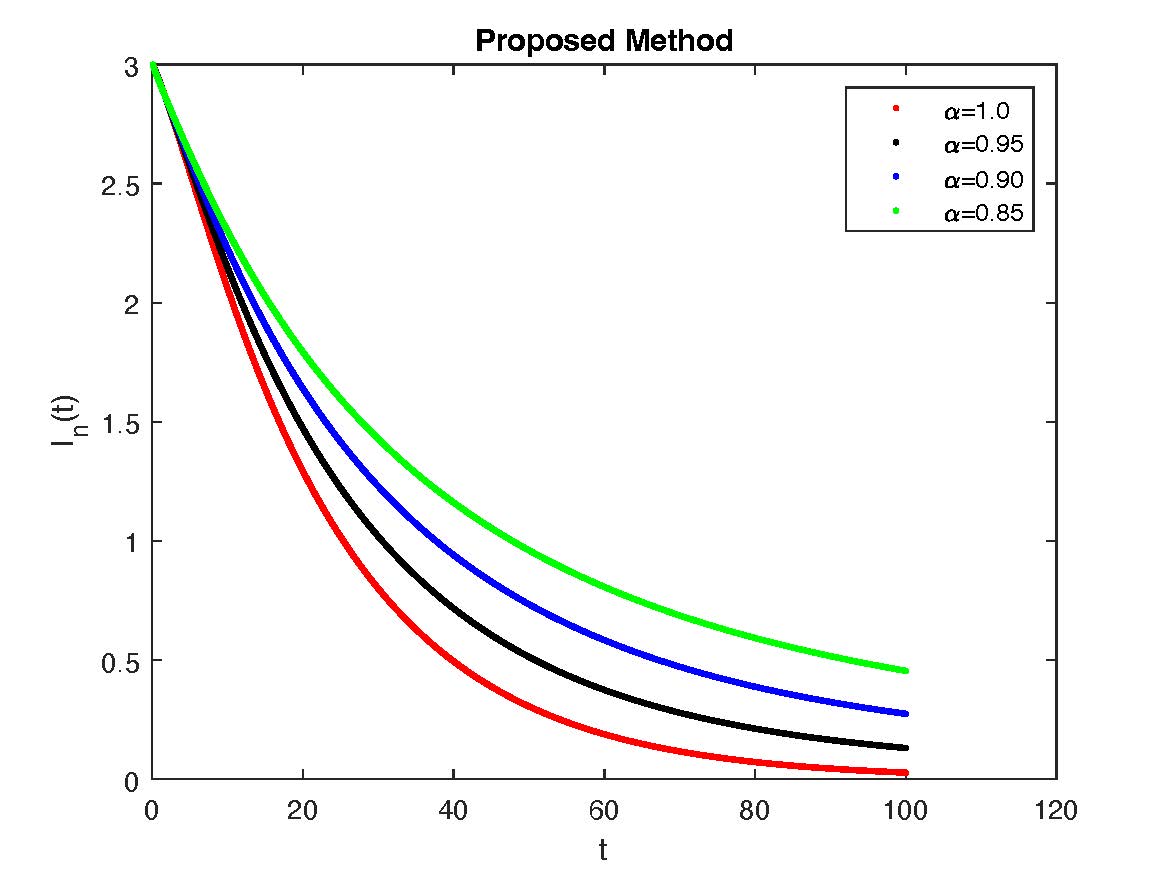}
	\includegraphics[width=6.5cm]{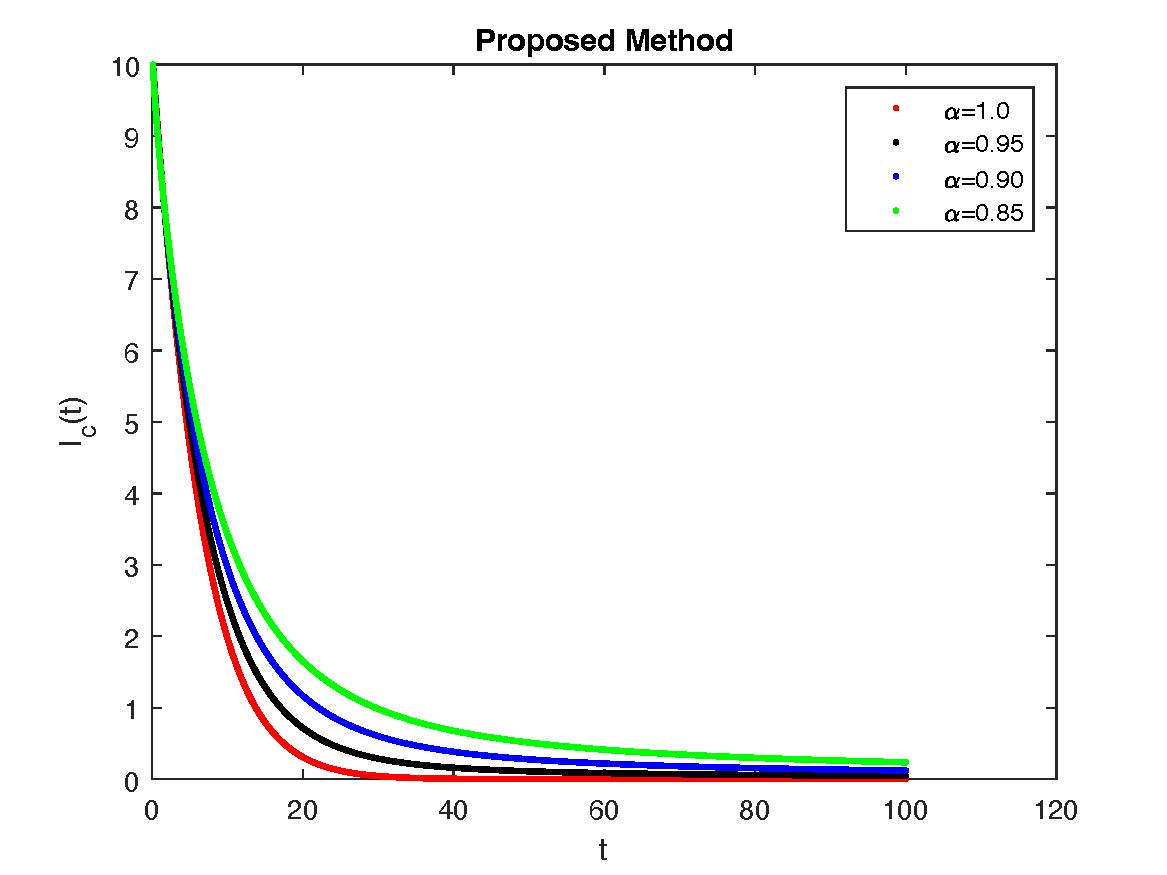}	
	\includegraphics[width=6.5cm]{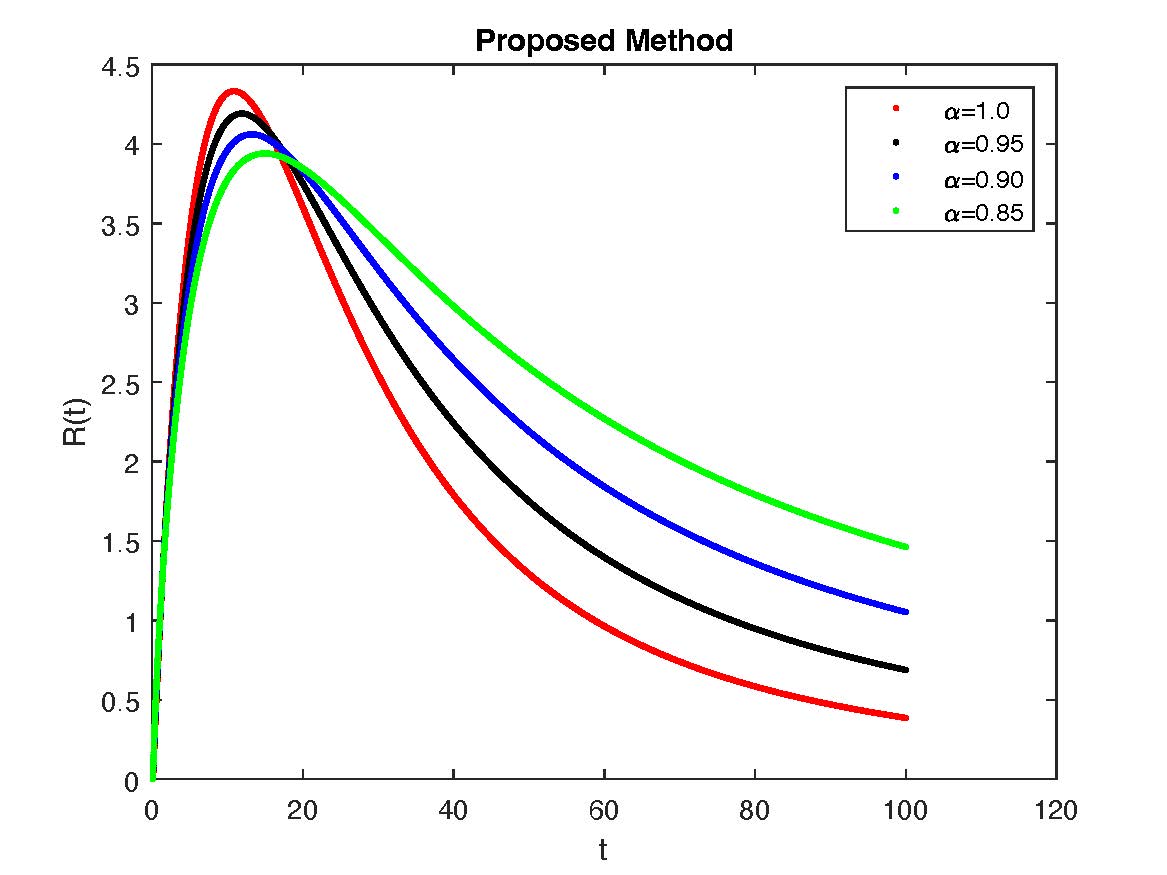}
	\includegraphics[width=6.5cm]{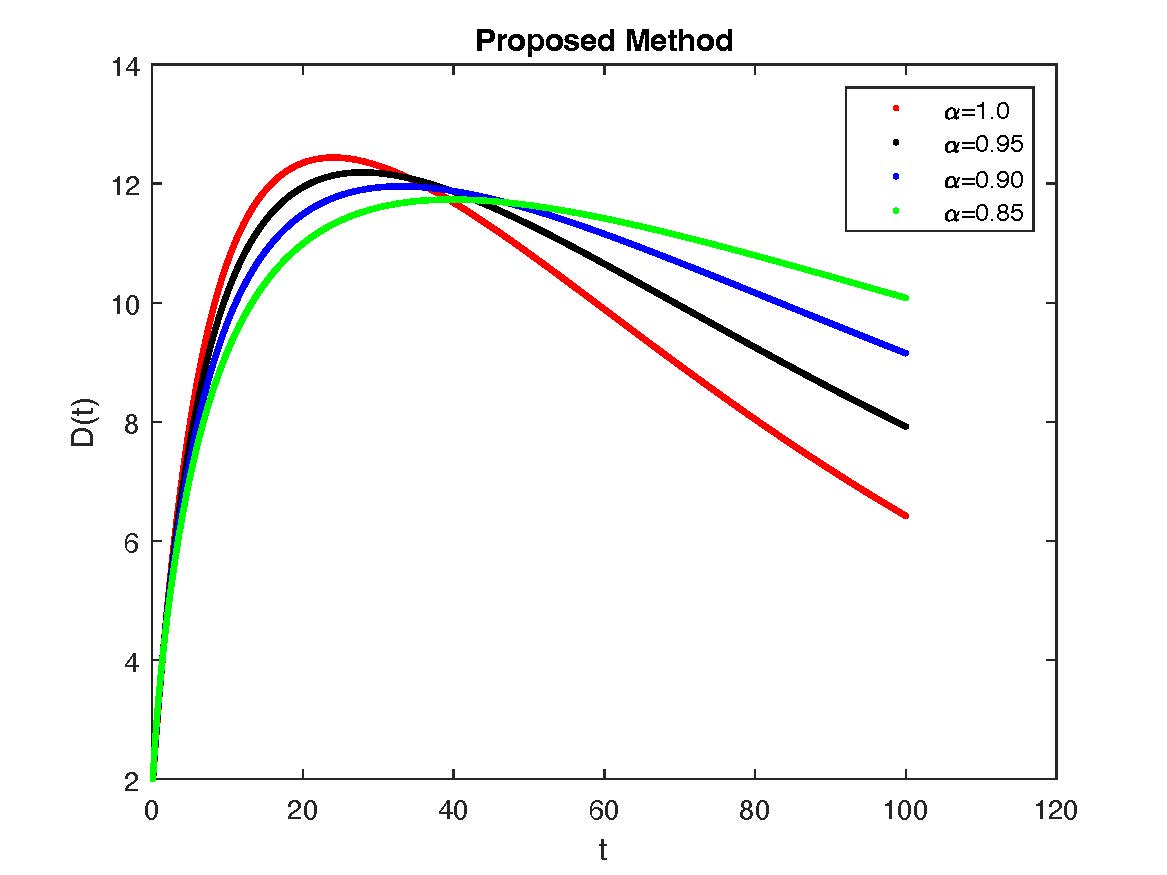}	
	\caption{Evolution of the population classes modeled with FFP: susceptible $S_p$, impacted by media $I$, positive opinion $I_p$, negative opinion $I_n$, confused $I_c$, overcome the misinformation $R$, death or denials $D$.}
\end{figure}

\begin{figure}[H]
	\centering
	\includegraphics[{width=7.5cm}]{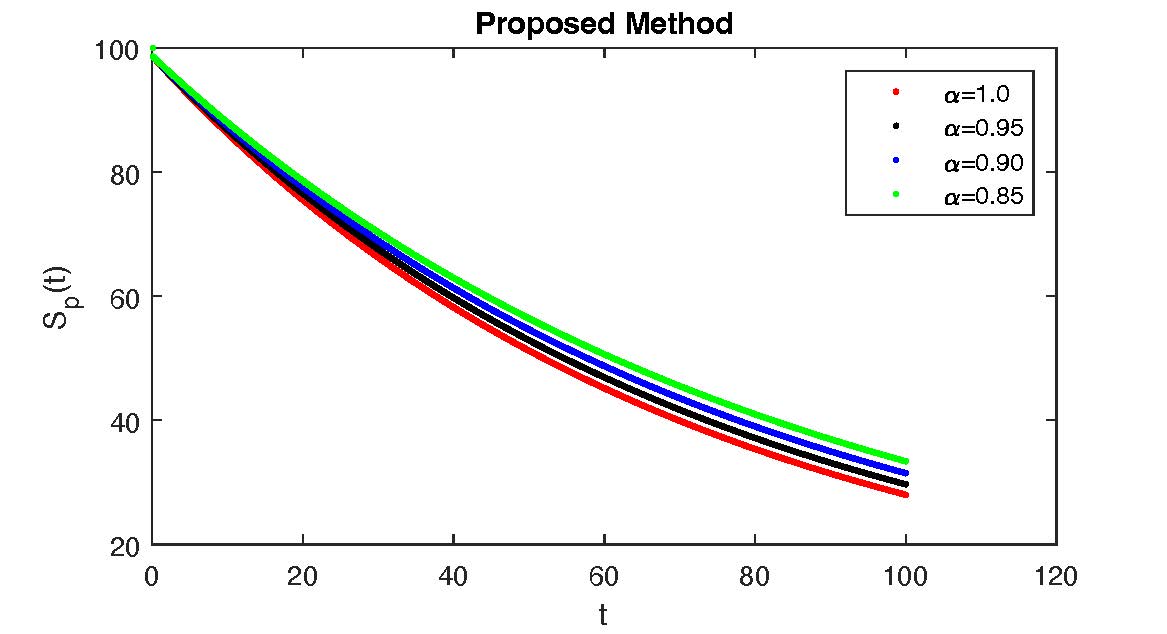}
	\includegraphics[{width=7.5cm}]{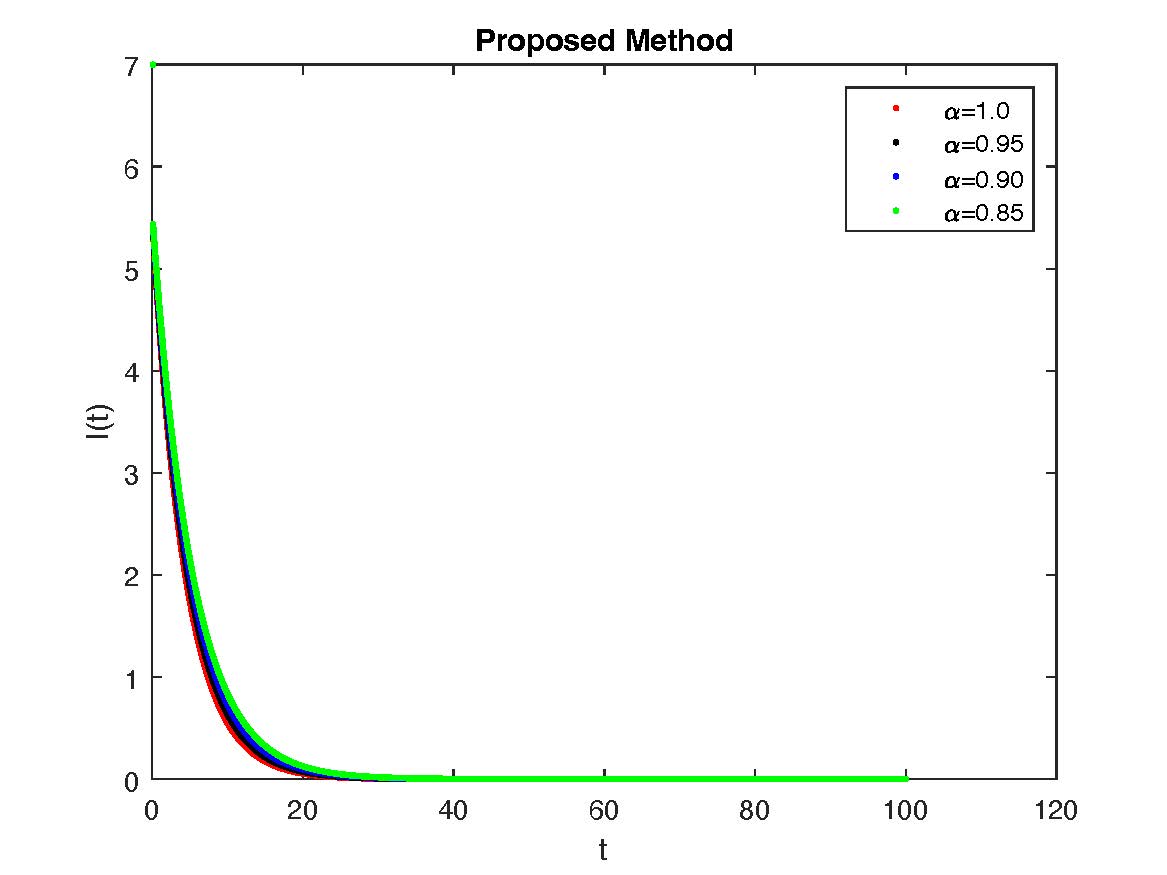}\\
	\includegraphics[{width=7.5cm}]{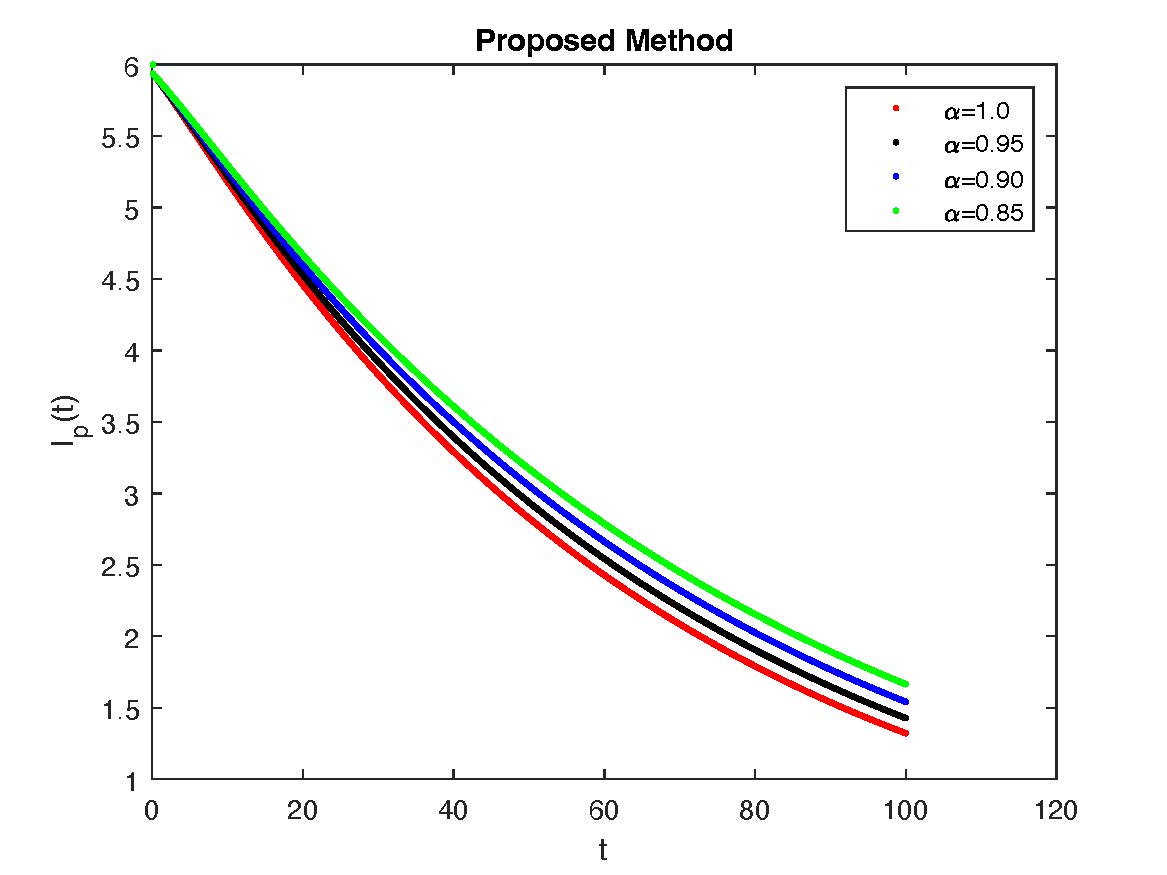}
	\includegraphics[{width=7.5cm}]{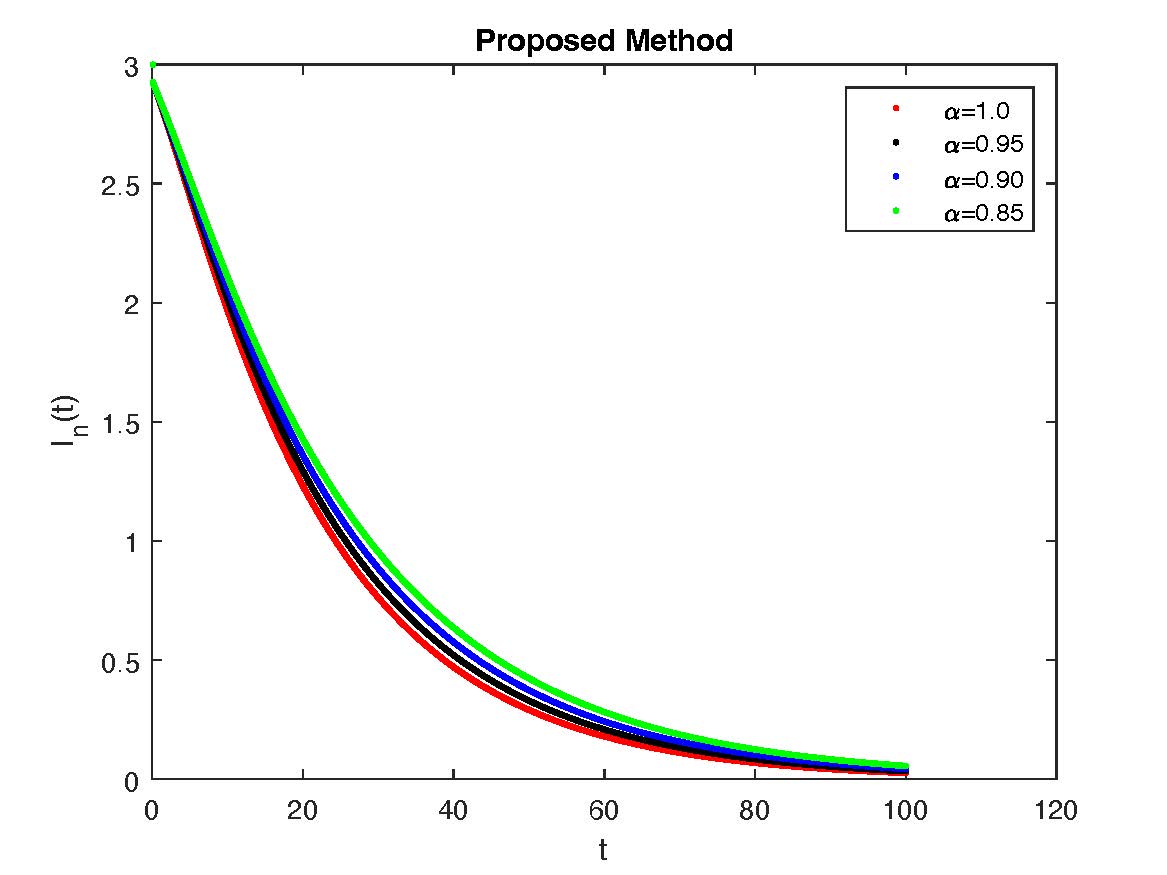}\\
	\includegraphics[{width=7.5cm}]{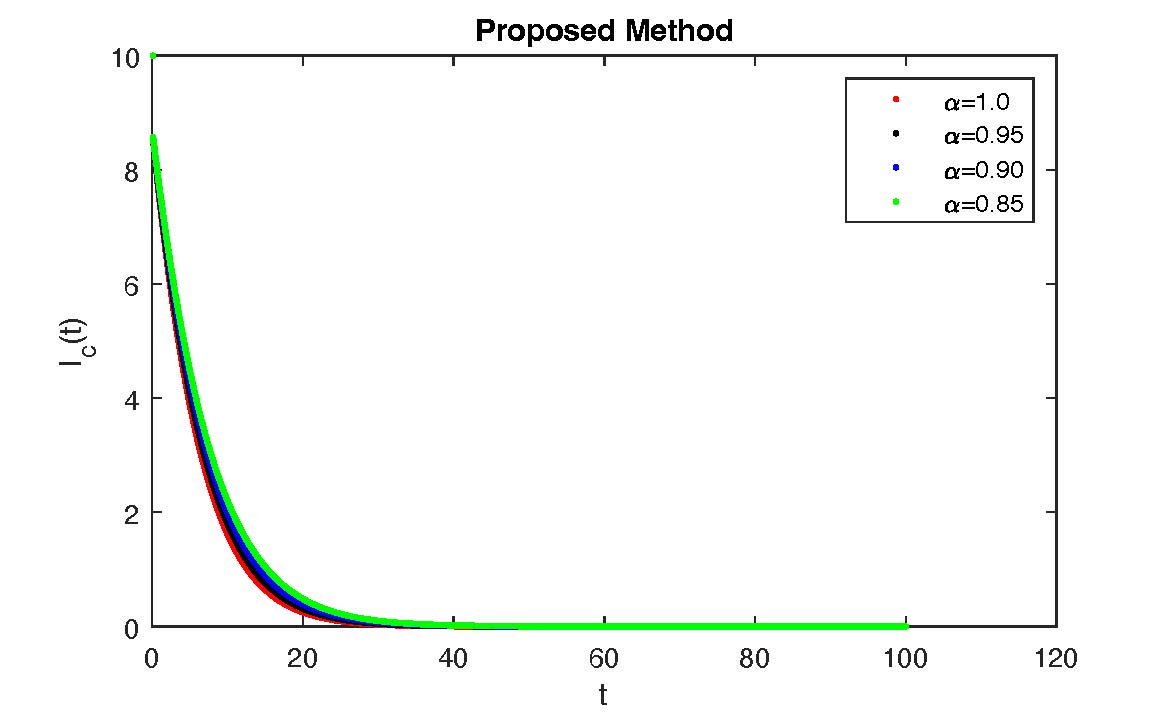}	
	\includegraphics[{width=7.5cm}]{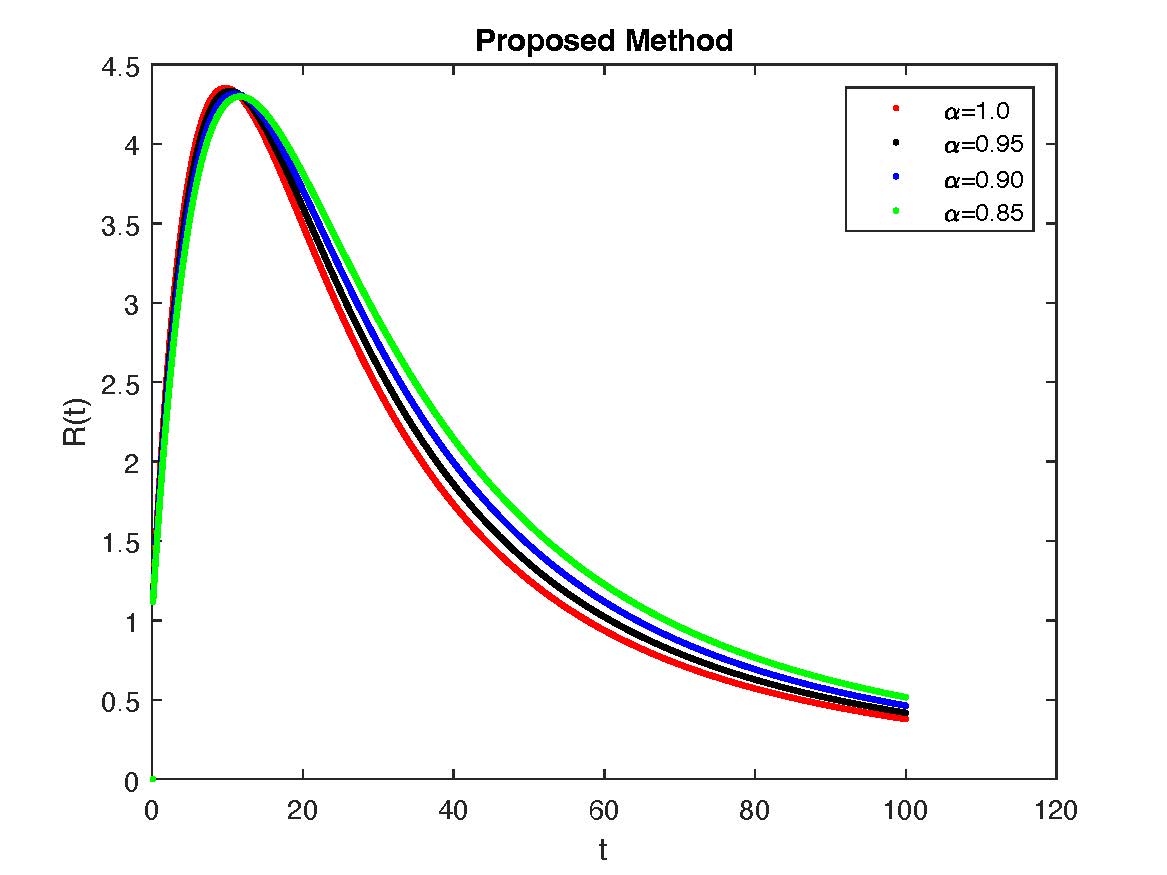}\\
	\includegraphics[{width=7.5cm}]{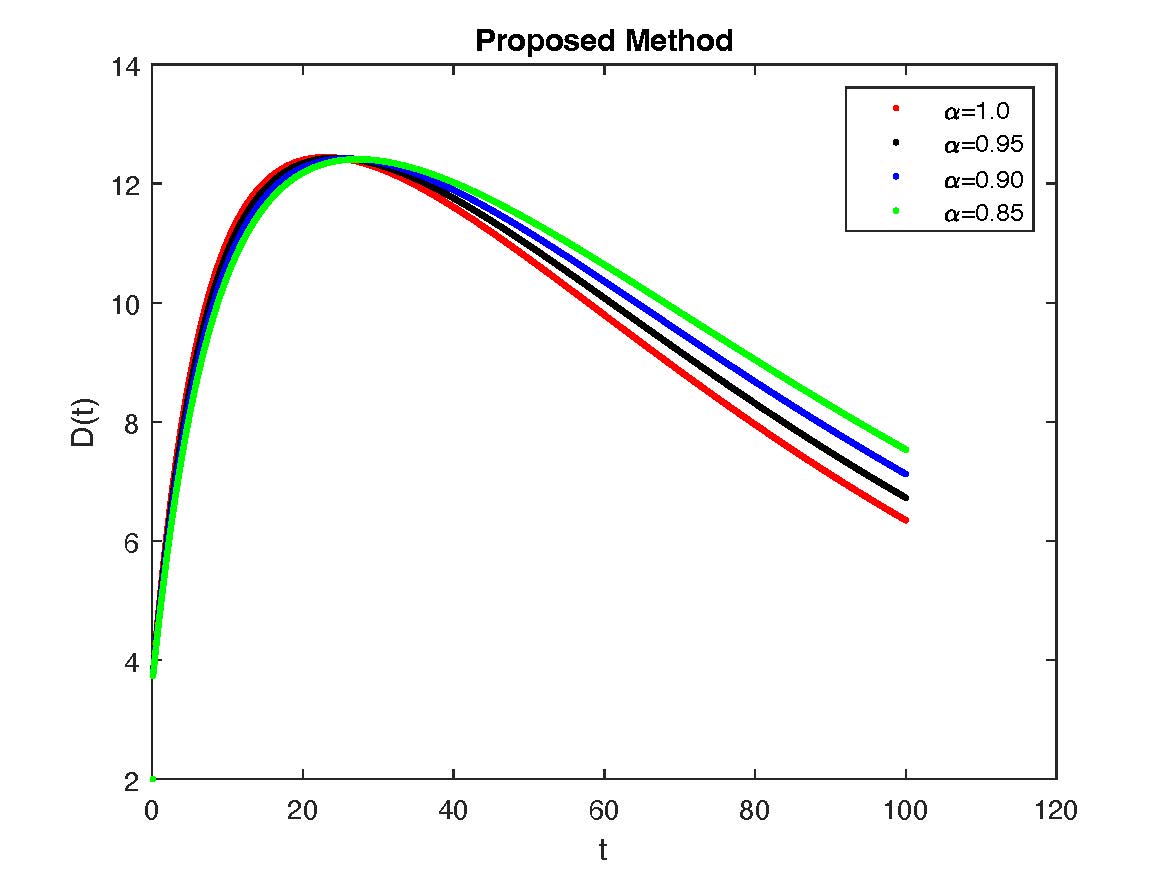}
	\caption{Evolution of the population classes modeled with FFE: susceptible $S_p$, impacted by media $I$, positive opinion $I_p$, negative opinion $I_n$, confused $I_c$, overcome the misinformation $R$, death or denials $D$.}
\end{figure}

\begin{figure}[H]
	\centering
	\includegraphics[{width=7.5cm}]{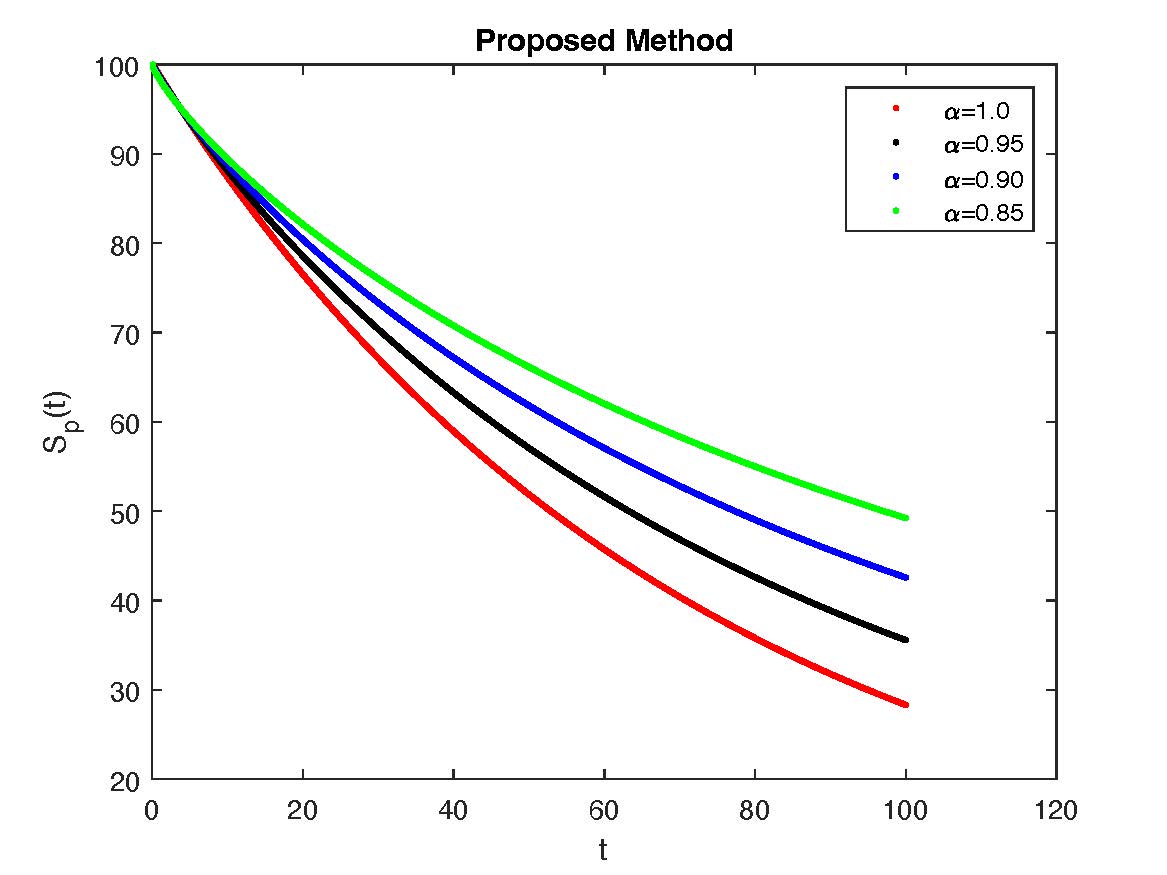}
	\includegraphics[{width=7.5cm}]{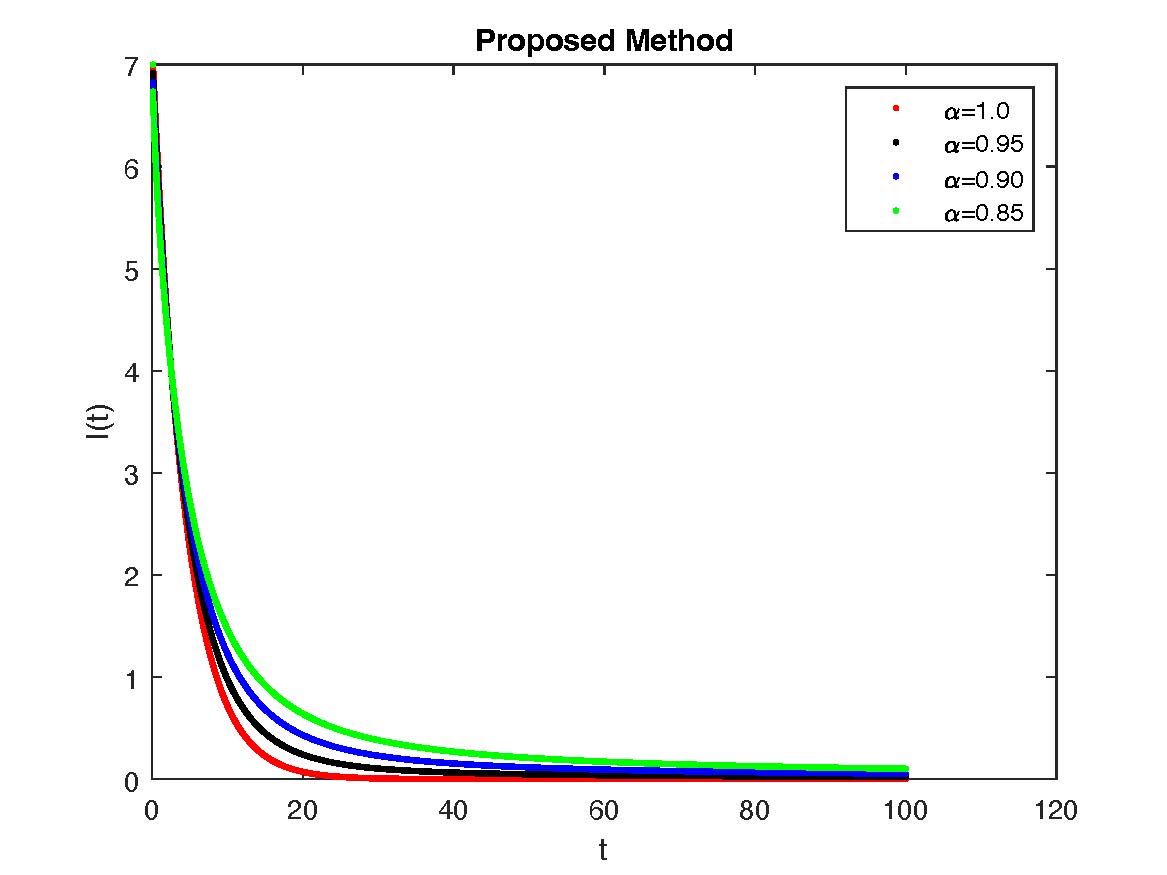}\\
	\includegraphics[{width=7.5cm}]{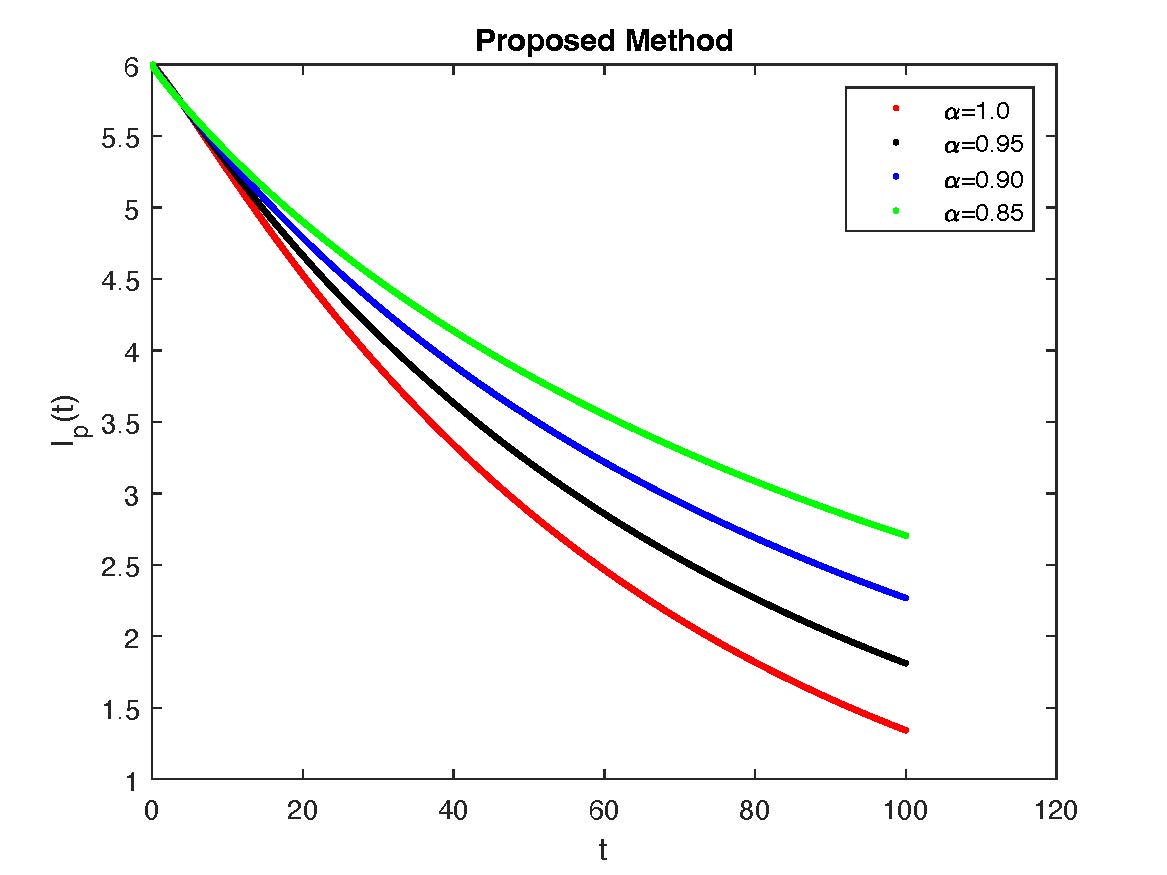}
	\includegraphics[{width=7.5cm}]{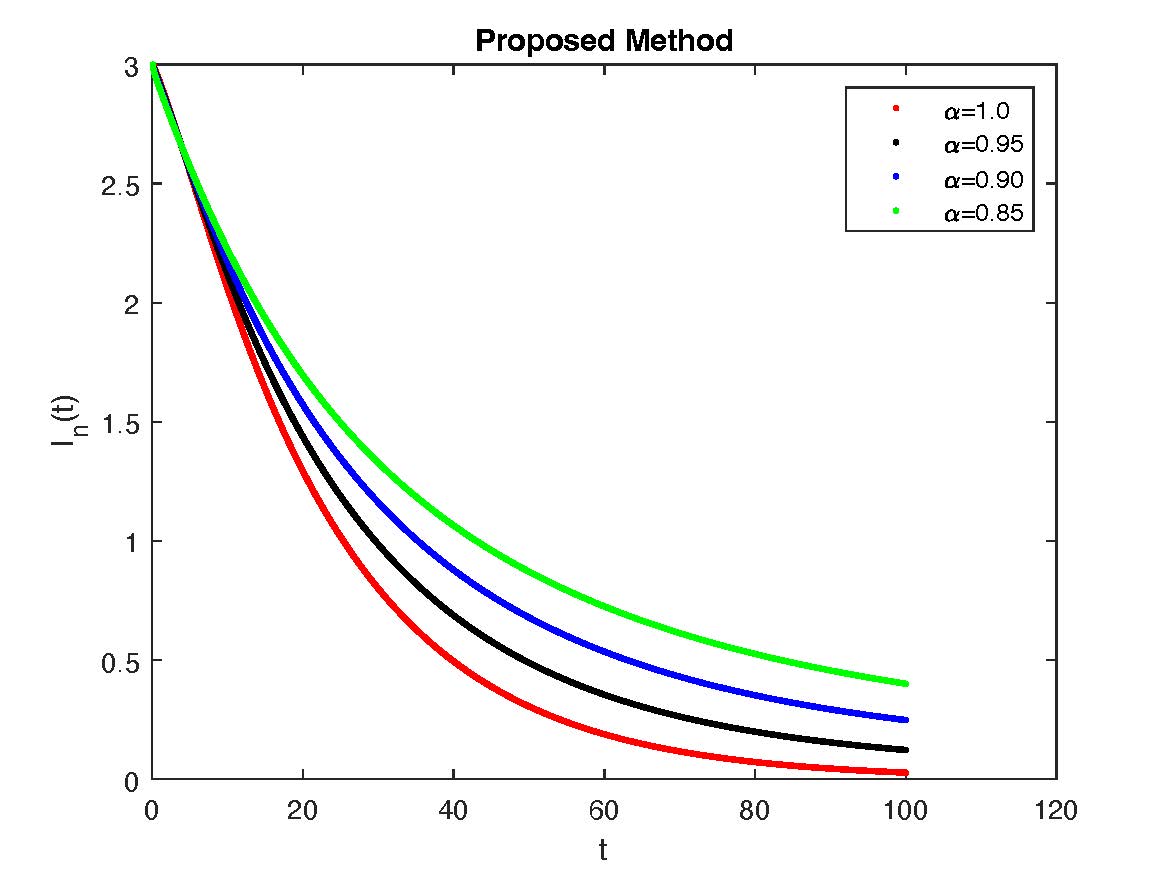}\\
	\includegraphics[{width=7.5cm}]{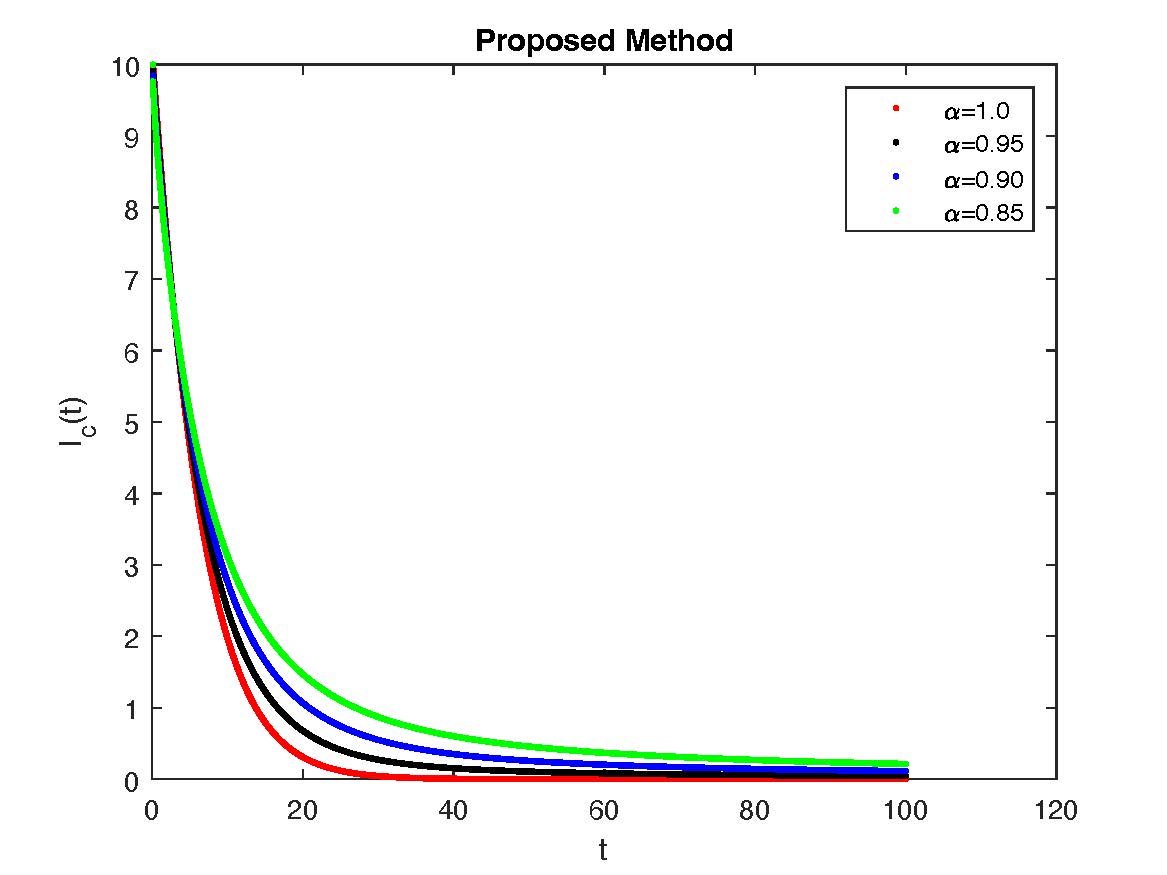}	
	\includegraphics[{width=7.5cm}]{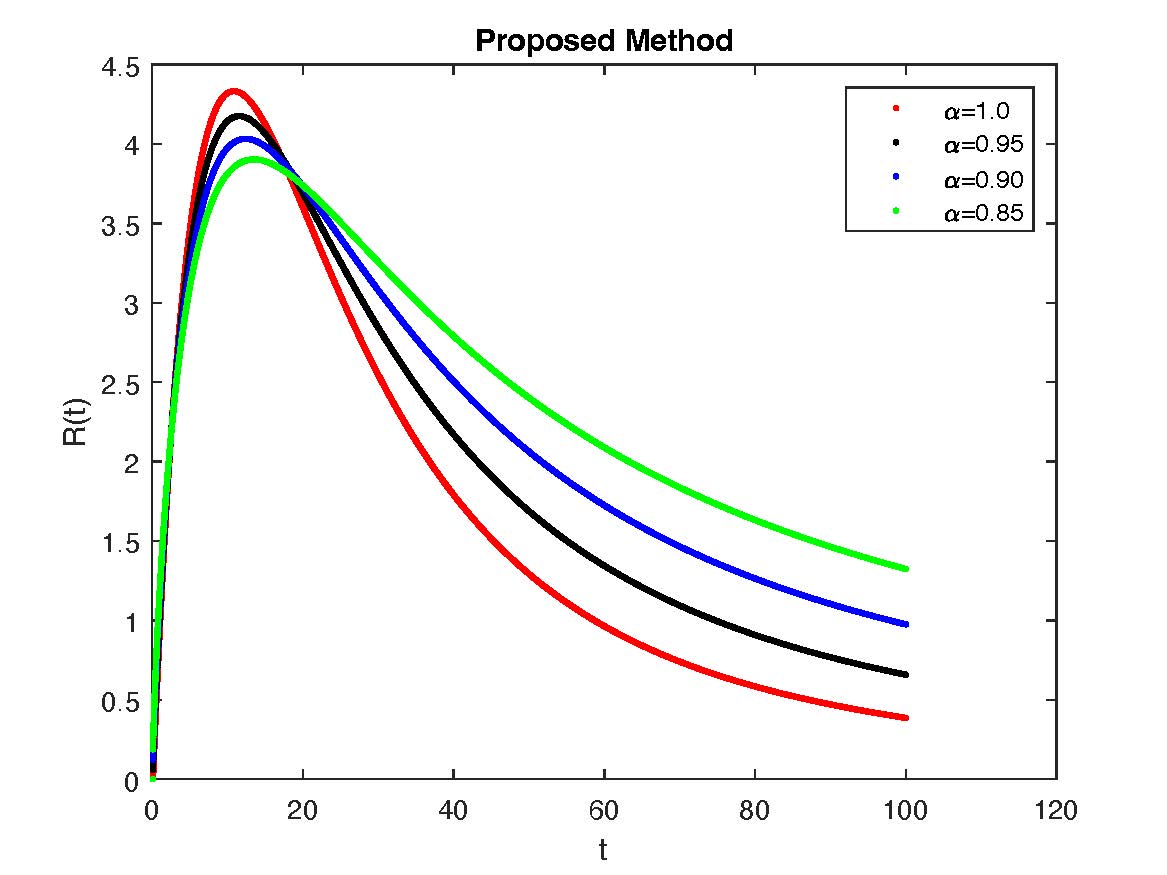}\\
	\includegraphics[{width=7.5cm}]{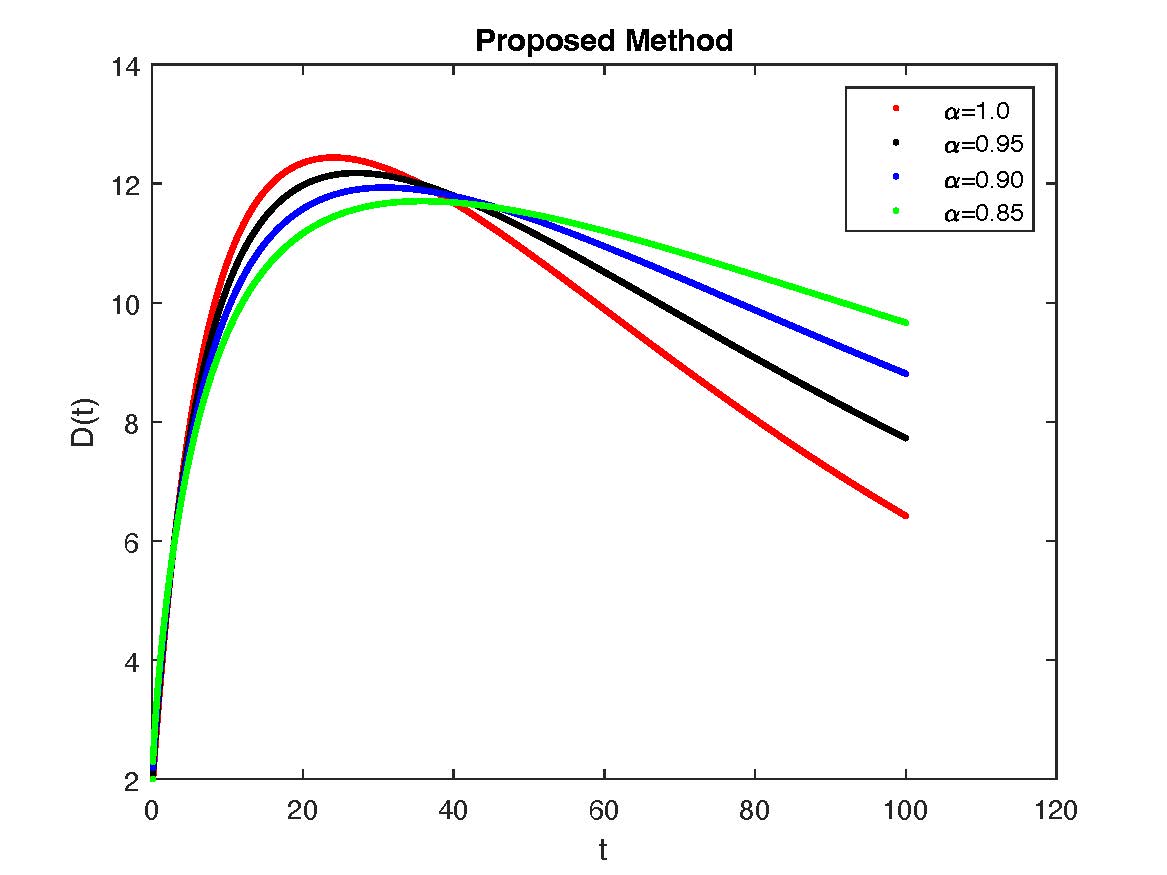}
	\caption{Evolution of the population classes modeled with FFM: susceptible $S_p$, impacted by media $I$, positive opinion $I_p$, negative opinion $I_n$, confused $I_c$, overcome the misinformation $R$, death or denials $D$.}
\end{figure}

\section{Conclusion}
\label{sec:conclusions}
In this paper, we have developed a mathematical model to depict the spreading and persistence of harmful information, inspired in epidemiological SIR models. Our model comprises seven ordinary differential equations, introducing a novel analysis that incorporates a strength number capturing the combined impact of both nonlinear and linear components within the epidemiological model. Moreover, we introduced the second derivative of the Lyapunov function to characterize the eventual detection of waves. To illustrate our findings, we conducted numerical simulations using three distinct kernels for the fractal fractional derivative.

\begin{flushleft}
	\textbf{Statement about data}
\end{flushleft}
No data was used in this work.
\begin{flushleft}
	\textbf{Conflict of Interests}
\end{flushleft}
The authors have no conflict of interests.

\begin{flushleft}
	\textbf{Acknowledgements}
\end{flushleft}
The third authos is supported by the R+D project ANDHI project CPP2021-008994 funded by MCIN/ AEI/10.13039/501100011033/ through European Union - NextGenerationEU/PRTR.

\bibliographystyle{alpha}
\bibliography{biblio}

\end{document}